\documentclass[12pt]{amsart}

\usepackage{amsthm,amsfonts,amsmath,amssymb,latexsym,epsfig}
\usepackage[usenames]{color}
\usepackage[all]{xy}

\DeclareMathAlphabet\oldmathcal{OMS}        {cmsy}{b}{n}
\SetMathAlphabet    \oldmathcal{normal}{OMS}{cmsy}{m}{n}
\DeclareMathAlphabet\oldmathbcal{OMS}       {cmsy}{b}{n} 

\usepackage{eucal}

\newtheorem{theorem}{Theorem}[section]
\newtheorem{lemma}[theorem]{Lemma}
\newtheorem{proposition}[theorem]{Proposition}
\newtheorem{corollary}[theorem]{Corollary}
\newtheorem{definition}[theorem]{Definition}
\newenvironment{remark}{\medskip \refstepcounter{theorem}
\noindent  {\bf Remark \thetheorem}.\rm}{\,}
\newtheorem{example}{Example}[section]
\newtheorem*{ack}{Acknowledgements}

\renewcommand{\thetheorem}{\thesection.\arabic{theorem}}

\def\<{\langle}

\def\>{\rangle}

\def\BOne{{\mathchoice {\rm 1\mskip-4mu l} {\rm 1\mskip-4mu l}
                          {\rm 1\mskip-4.5mu l} {\rm 1\mskip-5mu l}}}

\def\fract#1#2{\raise4pt\hbox{$ #1 \atop #2 $}}
\def\decdnar#1{\phantom{\hbox{$\scriptstyle{#1}$}}
\left\downarrow\vbox{\vskip15pt\hbox{$\scriptstyle{#1}$}}\right.}

\def\bbc{{\mathbb C}}

\def\bbn{{\mathbb N}}

\def\bbp{{\mathbb P}}
\def\bbq{{\mathbb Q}}
\def\bbr{{\mathbb R}}

\def\bbt{{\mathbb T}}

\def\bbz{{\mathbb Z}}

\def\gra{\alpha}
\def\grb{\beta}

\def\grg{\gamma}
\def\gri{\iota}
\def\grk{\kappa}
\def\grl{\lambda}

\def\gro{\omega}

\def\grs{\sigma}

\def\grO{\Omega}

\def\grS{\Sigma}

\def\bfa{{\bf a}}
\def\bfb{{\bf b}}

\def\bfr{{\bf r}}

\def\bfw{{\bf w}}

\def\cala{{\mathcal A}}

\def\calo{{\mathcal O}}
\def\calu{{\mathcal U}}
\def\cald{{\mathcal D}}
\def\cale{{\mathcal E}}

\def\calg{{\mathcal G}}
\def\calh{{\mathcal H}}
\def\cali{{\mathcal I}}

\def\calk{{\mathcal K}}

\def\calo{{\mathcal O}}

\def\cals{{\oldmathcal S}}

\def\calu{{\mathcal U}}
\def\calv{{\mathcal V}}
\def\calw{{\mathcal W}}

\def\la#1{\hbox to #1pc{\leftarrowfill}}
\def\ra#1{\hbox to #1pc{\rightarrowfill}}

\def\ge{{\mathfrak e}}

\def\gt{{\mathfrak t}}
\def\gu{{\mathfrak u}}

\def\gw{{\mathfrak w}}

\def\gz{{\mathfrak z}}

\def\gA{{\mathfrak A}}

\def\gC{{\mathfrak C}}

\def\gK{{\mathfrak K}}

\def\gR{{\mathfrak R}}
\def\gS{{\mathfrak S}}

\def\tpi{\tilde{\pi}}

\def\hook{\mathbin{\hbox to 6pt{%
                 \vrule height0.4pt width5pt depth0pt
                 \kern-.4pt
                 \vrule height6pt width0.4pt depth0pt\hss}}}

\title{Sasakian Geometry on Sphere Bundles} 
\author{Charles P. Boyer}
\author{Christina W. T{\o}nnesen-Friedman}

\thanks{The first author was partially supported by grant \#519432 from the Simons Foundation. The second author was partially supported by grant \#422410 from the Simons Foundation.}

\date{\today}
\address{Charles P. Boyer, Department of Mathematics and Statistics,
University of New Mexico, Albuquerque, NM 87131.}
\email{cboyer@math.unm.edu} 
\address{Christina W. T{\o}nnesen-Friedman, Department of Mathematics, Union
College, Schenectady, New York 12308, USA } \email{tonnesec@union.edu}

\begin{document}

\begin{abstract}
The purpose of this paper is to study the Sasakian geometry on odd dimensional sphere bundles over a smooth projective algebraic variety $N$ with the ultimate, but probably unachievable goal of understanding the existence and non-existence of extremal and constant scalar curvature Sasaki metrics. We apply the fiber join construction of Yamazaki \cite{Yam99} for K-contact manifolds to the Sasaki case. This construction depends on the choice of $d+1$ integral K\"ahler classes $[\gro_j]$ on $N$ that are not necessarily colinear in the K\"ahler cone. We show that the colinear case is equivalent to a subclass of a different join construction orginally described in \cite{BG00a,BGO06}, applied to the spherical case by the authors in \cite{BoTo13,BoTo14a} when $d=1$, and known as cone decomposable \cite{BHLT16}.
The non-colinear case gives rise to infinite families of new inequivalent cone indecomposable Sasaki contact CR on certain sphere bundles. We prove that the Sasaki cone for some of these structures contains an open set of extremal Sasaki metrics and, for certain specialized cases, the regular ray within this cone is shown to have constant scalar curvature. We also compute the cohomology groups of all such sphere bundles over a product of Riemann surfaces.
\end{abstract}

\maketitle

\markboth{Sasakian Geometry on Sphere Bundles}{C. P. Boyer and C. W. T{\o}nnesen-Friedman}

\tableofcontents

\section{Introduction}
In 1999 T. Yamazaki \cite{Yam99} gave a construction of K-contact structures on certain odd dimensional sphere bundles over a symplectic manifold. It is easy to see that if one takes the base symplectic manifold to be a smooth projective algebraic variety $N$, Yamazaki's construction gives a family of Sasakian structures on these odd dimensional sphere bundles. More precisely they are the unit sphere bundles of a complex vector bundle that splits as a sum of line bundles. Such structures have recently become of interest in the study of Calabi-Yau $A_\infty$ algebras used in topological conformal field theory \cite{Cos07}. In particular, it was shown in  \cite{TaTs17} that the Calabi-Yau $A_\infty$ algebras studied in \cite{TsTsYau16} are equivalent to the standard de Rham differential graded algebra on certain odd dimensional sphere bundles. For this among other reasons, we feel that a concerted study of Sasaki geometry on sphere bundles is warranted. 

In Theorem 4.8 of \cite{BHLT16} a partial classification under a fairly restrictive condition was given in the case of $S^3$ bundles over a smooth projective algebraic variety. Moreover, several examples where the restrictive condition is violated were given some of which used Yamazaki's fiber join construction. These conditions involve invariants known as cone decomposability. It is the purpose of the present paper to give a further study of this decomposability issue within the context of Yamazaki's construction for Sasakian structures on odd dimensional sphere bundles. An ultimate (perhaps unachievable) goal is to classify the Sasaki CR structures on such sphere bundles as well as understand the existence or non-existence of extremal and constant scalar curvature Sasaki metrics on such sphere bundles. This paper takes the first few steps toward this goal. 

The fiber join construction requires choosing $d+1$ not necessarily distinct K\"ahler classes in the K\"ahler cone $\calk(N)$ of a smooth projective algebraic variety $N$ which produces a Sasakian structure on an $S^{2d+1}$ sphere bundle over $N$. We can divide these into two types: 1. all the K\"ahler classes are colinear; 2. not all K\"ahler classes are colinear. We show that this dichotomy corresponds to the notions of {\it cone decomposable} and {\it cone indecomposable} introduced in \cite{BHLT16} and which are invariants of the underlying Sasaki CR structure. In Section 3 we prove Proposition \ref{projprod} that the colinear fiber joins are equivalent to special cases of the Sasaki joins described in \cite{BGO06}. We refer to the latter as {\it regular} Sasaki joins. They have been studied in detail in \cite{BoTo14a} when $d=1$; however, here we also present some new results when $d>1$. It is easy to see that if $(N,\omega_N)$ is an extremal K\"ahler structure the colinear case always has an open set of extremal Sasaki metrics in its Sasaki cone. Our main focus will thus be on the non-colinear or cone indecomposable case where our proof of extremality is not a priori, but depends on the applicability of the admissible construction as described in \cite{ACGT08}. This requeires a splitting of the positive integer $d=d_0+d_\infty$. For the definition of admissible and super admissible see Section \ref{admsect}.

\begin{theorem}\label{gennoncothm}
Let $M_\gw$ be a super admissible cone indecomposable fiber join whose regular quotient is a ruled manifold of the form $\bbp(E_0\oplus E_\infty)\longrightarrow N$ where $E_0 ,E_\infty$ are projectively flat hermitian holomorphic vector bundles on $N$ of complex dimension $(d_0+1),(d_\infty+1)$ respectively, and $N$ is a local K\"ahler product of non-negative CSC metrics. Then the Sasaki cone of $M_\gw$ has an open set of extremal Sasaki metrics.
\end{theorem}

It is possible to be more explicit and also to obtain CSC Sasaki metrics as well to relax the non-negativity assumption. However, to do so we need to understand the admissible polynomial, whose degree grows linearly with $d$, in more detail. 
Therefore, we do the analysis for some special cases of sphere bundles over certain products of Riemann surfaces.

\begin{theorem}\label{noncothm}
Let $M_K$ be a cone indecomposable admissible fiber join over $\grS_{g_1}\times \grS_{g_2}$. 
Then the Sasaki cone of $M_K$ has an open set of extremal Sasaki metrics containing a regular ray generated by $\xi_1$ 
\begin{enumerate}
\item for all $d\geq 1$ and for all 2 by 2 matrices $K$ with values in $\bbz^+$ if $g_1=0, g_2=0,1$;
\item if $g_1=0, g_2=g>1$, $d_0=d_\infty=1$ and  
\begin{equation*}\label{specextrmatrix}
K= \begin{pmatrix}
     2 & g \\
     1 & 1
\end{pmatrix};
\end{equation*}
\item with $d_0=d_\infty=0$, $g_1=g_2=g$, and $\xi_1$ having constant scalar curvature
\begin{enumerate}
\item if $g=0$ and $K= \begin{pmatrix}
    k & l \\
l & k
\end{pmatrix}$
with $k,l\in \bbz^+$ and $k\neq l$, or
\item $g=1$,
and $K= \begin{pmatrix}
     k^1_0 & k^2_0 \\
k^1_{\infty} & k^2_{\infty}
\end{pmatrix}$ is such that 
$\frac{k^1_0-k^1_\infty}{k^1_0+k^1_\infty}=\frac{k^2_\infty-k^2_0}{k^2_0+k^2_\infty}$, or
\item $g>1$,
if $K= \begin{pmatrix}
     k+1 & k \\
k & k+1
\end{pmatrix}$ is such that $k\in \bbz^+$ satisfies
$k> \lfloor\frac{2 g-3 +\sqrt{4 g^2-8 g+5}}{2} \rfloor$. 
\end{enumerate}
\end{enumerate}
\end{theorem}

In the case of a product of Riemann spheres more can be said. For example we can prove the existence of a countable infinity of inequivalent Sasaki contact structure on the same (up to diffeomorphism)  sphere bundle.

\begin{theorem}\label{gen0prodcor}
Consider the set $\{M_K\}$ of cone indecomposable admissible fiber joins over $\bbc\bbp^1\times \bbc\bbp^1$ in case (1) of Theorem \ref{noncothm}. There exist at least one diffeomorphism type within this set that admits a countable infinity of  inequivalent Sasaki contact structures all of which have an open set of extremal Sasaki metrics in their Sasaki cone.
\end{theorem}

\begin{remark} A similar result holds in the cone decomposable case.
\end{remark}

In Section \ref{rsprod} we describe the topology of the manifolds $M_K$. If $d_0$ and $d_\infty$ are both odd, $M_K$ is a spin manifold for any $K$, otherwise, both spin and non-spin can occur. In Proposition \ref{k>1topprop} we compute the cohomology groups of a general Sasaki fiber join on $S^{2d+1}$ bundles over the product of Riemann surfaces $\grS_{g_1}\times \grS_{g_2}$. More can be said when $g_1=g_2=0$. For case (a) of Theorem \ref{noncothm} we prove in Proposition \ref{g0homeo} that each ordered pair $(k,l)$ of positive integers with $k>l$ determines a unique homeomorphism type of $S^3$ bundle over $\bbc\bbp^1\times \bbc\bbp^1$, so there exists a countable infinity of such homeomorphism types. Furthermore, in each case the diffeomorphism type is known up to finite ambiguity.
Since $\bbc\bbp^1\times \bbc\bbp^1$ has a 2-torus of Hamiltonian automorphisms, all the $M_K$ of type (1) are toric, and of Koiso-Sakane type, that is, the quotient manifold of the regular ray is the Bott manifold $$M_3(0,k-l,-(k-l))= \bbp\bigl(\BOne\oplus \calo(k-l,-(k-l))\bigr).$$ 
The K\"ahler geometry of these quotients has been studied in detail elsewhere (see for example \cite{BoCaTo17} and references therein).  

Also in Section 5.4 we present some more extremality results in special cases when $g_1\neq g_2$ as well as cases with negative constant transverse scalar curvature. Finally in Section 5.5 we present the existence results of extremal Sasaki metrics for special cases with higher dimensional base space $N$.

\begin{ack}
We thank Eveline Legendre and Hongnian Huang for their interest in our work.
\end{ack}

\section{Brief Review of Sasaki Geometry}
Recall that a Sasakian structure on a contact manifold $M^{2n+1}$ of dimension $2n+1$ is a special type of contact metric structure $\cals=(\xi,\eta,\Phi,g)$ with underlying almost CR structure $(\cald,J)$ where $\eta$ is a contact form such that $\cald=\ker\eta$, $\xi$ is its Reeb vector field, $J=\Phi |_\cald$, and $g=d\eta\circ (\BOne \times\Phi) +\eta\otimes\eta$ is a Riemannian metric. $\cals$ is a {\it K-contact} structure if $\xi$ is a Killing vector field and it is {\it Sasakian} if in addition the almost CR structure is integrable, i.e. $(\cald,J)$ is a CR structure. We refer to \cite{BG05} for the fundamentals of Sasaki geometry. We call $(\cald,J)$ a  {\it Sasaki CR structure} or {\it CR structure of Sasaki type}, and $\cald$ a {\it Sasaki contact structure} or {\it contact structure of Sasaki type}. We shall always assume that the Sasaki manifold $M^{2n+1}$ is compact and connected.

\begin{definition}\label{SasCRequiv}
Let $(M,\cald,J)$ and $(M',\cald',J')$ be Sasaki CR structures. We say that $(M,\cald,J)$ and $(M',\cald',J')$ are {\bf equivalent}, denoted $(M',\cald',J')\approx(M,\cald,J)$, if there exists a diffeomorphism $\psi:M\longrightarrow M'$ such that 
$$\psi_*\cald=\cald', \qquad J'=\psi_* J \psi_*^{-1}.$$
\end{definition}

\subsection{Invariants and the Classification of Sasaki CR Structures}
The classification of Sasaki CR structures on a given manifold is of major importance to us. An important invariant is the conical family of Sasakian structures within a fixed contact CR structure $(\cald,J)$ known as the {\it (unreduced) Sasaki cone} and denoted by $\gt^+$. We are also interested in a variation within this family. To describe the Sasaki cone we fix a Sasakian structure $\cals_o=(\xi_0,\eta_o,\Phi_o,g_o)$ on $M$ whose underlying CR structure is $(\cald,J)$ and let $\gt$ denote the Lie algebra of the maximal torus in the automorphism group of $\cals_o$. The {\it (unreduced) Sasaki cone} \cite{BGS06} is defined by
\begin{equation}\label{sascone}
\gt^+(\cald,J)=\{\xi\in\gt~|~\eta_o(\xi)>0~\text{everywhere on $M$}\},
\end{equation}
which is a cone of dimension $k\geq 1$ in the Lie algebra $\gt$. When the underlying CR structure $(\cald,J)$ is understood we often write $\gt^+$ instead of $\gt^+(\cald,J)$. The {\it reduced Sasaki cone} $\grk(\cald,J)$ is $\gt^+(\cald,J)/\calw$ where $\calw$ is the Weyl group of the maximal compact subgroup of $\gC\gR(\cald,J)$ which is viewed as the {\it moduli space} of Sasakian structures with underlying CR structure $(\cald,J)$, and it is an important invariant of the Sasaki CR structure, that is if $(\cald',J')\approx(\cald,J)$ are equivalent Sasaki CR structures, then $\grk(\cald,J)=\grk(\cald',J')$. In particular, the reducibility structure \cite{BHLT16} of $\grk(\cald,J)$ which gives rise to a multi-foliate structure on the manifold $M$. This multifoliate structure is an invariant, up to order, of the CR Sasaki structure. On the tangent space level it corresponds to decomposing $\cald$ into irreducible subspaces up to order. A cruder invariant is the sequence of ranks of the individual pieces. However, in the present paper it is enough to consider just two types of Sasaki CR structures, namely, cone decomposable and cone indecomposable. Examples of such Sasaki CR structures were given in Section 4 of \cite{BHLT16}. 

Another important invariant of $(\cald,J)$ is its first Chern class $c_1(\cald)$. In fact it is an invariant of the underlying contact structure. So if the contact structures $\cald$ and $\cald'$ are equivalent, then $c_1(\cald)=c_1(\cald')$. The converse is not true. There are more subtle invariants related to `contact homology' which were explored to some extent for Sasaki contact structures in \cite{BoPa10,BMvK15}, but will not concern us here.

In practice it is more convenient to work with the unreduced Sasaki cone $\gt^+(\cald,J)$. It is also clear from the definition that $\gt^+(\cald,J)$ is a cone under the transverse scaling defined by
\begin{equation}\label{transscale}
\cals=(\xi,\eta,\Phi,g)\mapsto \cals_a=(a^{-1}\xi,a\eta,g_a),\quad g_a=ag+(a^2-a)\eta\otimes\eta, \quad a\in\bbr^+
\end{equation}
So Sasakian structures in $\gt^+$ come in rays, and since the Reeb vector field $\xi$ is Killing $\dim\gt^+\geq 1$, and it follows from contact geometry that $\dim\gt^+\leq n+1$. When $\dim\gt^+(\cald,J)=n+1$ we have a toric contact manifold of Reeb type studied in \cite{BM93,BG00b,Ler02a,Ler04,Leg10,Leg16}. In this case there is a strong connection between the geometry and topology of $(M,\cals)$ and the combinatorics of $\gt^+(\cald,J)$. Much can also be said in the complexity 1 case ($\dim\gt^+(\cald,J)=n$) \cite{AlHa06}.

Recall \cite{BGS06} that a Sasakian structure $(\xi,\eta,\Phi,g)$ is {\it extremal} if the $(1,0)$ gradient of the scalar curvature is transversely holomorphic. A main focus of the present paper is existence proofs of extremal and constant scalar curvature Sasaki metrics in certain cone indecomposable CR Sasaki structures on sphere bundles. The essential technique is the admissible construction described in \cite{ACGT08}. We shall often use the openness theorem of \cite{BGS06} which says that if an extremal Sasaki metric exists, there is an open set of extremal Sasaki metrics in the Sasaki cone $\gt^+(\cald,J)$. We emphasize that both extremal and constant scalar curvature Sasaki metrics come in rays; however, it is the transverse scalar curvature $s^T_g=s_g+2n$, which is the scalar curvature of the transverse K\"ahler metric $g^T$, that rescales not the scalar curvature $s_g$ of the Sasaki metric $g$. Under scaling the latter transforms as $s_{g_a}=a^{-1}(s_g+2n)-2n$.

\subsection{Sphere Bundles}
In this paper we are interested in $2d+1$ dimensional sphere bundles\footnote{By a sphere bundle we mean an oriented sphere bundle with the linear structure group $SO(d+1)$.} $M$ over a smooth projective algebraic variety $N$ of complex dimension $n$ such that a Sasakian structure on $M$ restricts to a weighted Sasakian structure on each fiber which is the standard sphere $S^{2d+1}$:
\begin{equation}\label{sphbun}
\begin{matrix} S^{2d+1} &\longrightarrow &M \\
&& \decdnar{} \\
&& N.
\end{matrix}
\end{equation}
The dimension of the Sasaki manifold $M$ is $2n+2d+1$, and note that if $N$ is toric, so is $M$ with $\dim \gt^+(\cald,J)=n+d+1$. More generally, the complexity of $M$ equals the complexity of $N$, and we have 
\begin{equation}\label{dimsascone}
d+1\leq \dim\gt^+(\cald,J)\leq n+d+1.
\end{equation}
It is well known that the Sasaki cone of the standard Sasaki CR structure on $S^{2d+1}$ is the $(d+1)$-st orthant, and this defines a subcone of $\gt^+$ of dimension $d+1$ which we denote by $\gt^+_{sph}(\cald,J)$. The sphere bundles studied in this paper all have a regular Reeb vector field $\xi_1\in\gt^+_{sph}(\cald,J)$ which plays an important role for us, and whose quotient is described in more detail in Section \ref{orbquot} below. However, in contrast there are sphere bundles with CR Sasaki structures having no regular Reeb field in their Sasaki cone. In \cite{BG05h} infinitely many such Sasaki CR structures are given on the trivial sphere bundles $S^{2d}\times S^{2d+1}$ for $d>1$ which have Sasaki metrics of  positive Ricci curvature . They are represented by Brieskorn manifolds belonging to infinitely many inequivalent contact structures \cite{BMvK15} and it is shown in \cite{BovC16} that their Sasaki cones admit no extremal Sasaki metrics whatsoever.

We briefly discuss some topology of sphere bundles over smooth projective algebraic varieties. From the long exact homotopy sequence of the bundle \eqref{sphbun} we have 
\begin{equation}\label{homotopy}
\pi_i(M)\approx \pi_i(N)~ \text{for $i\leq 2d$}. 
\end{equation}
Yamazaki treated the special case of sphere bundles over a Riemann surface $\grS_g$ of genus $g$. In this case there are precisely two diffeomorphism types, namely the trivial bundle $\grS_g\times S^{2d+1}$, and the nontrivial bundle $\grS_g\widetilde{\times} S^{2d+1}$. These are distinguished by the second Stiefel-Whitney class of the complex vector bundle $E$.

Consider the Leray-Hirsch Theorem of the fibration \eqref{sphbun}. The condition that there is a global $(2d+1)$-class that restricts to the fundamental class of $S^{2d+1}$ implies the vanishing of the Euler class of the bundle. 
Then from the Leray-Serre spectral sequence of \eqref{sphbun} we get

\begin{proposition}\label{fibjointop}
Let $N$ be a compact symplectic manifold and $M$ an $S^{2d+1}$ bundle over $N$. Then 
$$H^p(M,\bbz)\approx H^p(N,\bbz)$$ 
for all $p< 2d+1$. In particular, if $n=\dim_\bbc N\leq d$ then $M$ has the integer cohomology groups of the product $N\times S^{2d+1}$.
\end{proposition}

\begin{remark}\label{grpnotring}
In the last statement of Proposition \ref{fibjointop} we have an isomorphism of groups, but not necessarily an isomorphism of rings. It is an interesting and important problem to determine the cohomology ring structure.
\end{remark}

We are also interested in the possible diffeomorphism types of $M$ for which we apply Sullivan's rational homotopy theory \cite{Sul77} to our sphere bundles \eqref{sphbun}.
If $N$ is simply connected\footnote{More generally, we can assume that $\pi_1(N)$ is nilpotent and that it acts nilpotently on the higher homotopy groups of $N$. Then $N$ is said to be a nilpotent space.} the rational homotopy type of $M$ is well understood, see Example 2.69 in \cite{FeOpTa08}. For the sphere bundle \eqref{sphbun} Sullivan's relative minimal model is
\begin{equation}\label{rathomsphbun}
(\wedge V,d)\longrightarrow (\wedge V\otimes u,d)\longrightarrow (\wedge u,0)
\end{equation}
where $(\wedge V,d)$ is a differential commutative graded algebra (cdga) that is a model for $N$, $(\wedge u,0)$ a model for $S^{2d+1}$, and $du=\ge$ is a cocyle in $(\wedge V)^{2d+2}$ representing the Euler class of the sphere bundle $M$. The minimal model is completely determined by the Euler class. Recall that a cdga is said to be {\it formal} if there exists a quasi-isomorphism\footnote{that is a morphism of cdga's that induces an isomorphism in cohomology.}
$(\wedge V,d)\approx (H^*(M,\bbq),0)$. Then

\begin{lemma}\label{forlem}
A sphere bundle $M$ over a nilpotent base $N$ is formal when $n=\dim_\bbc N\leq d$.
\end{lemma}

\begin{proof}
When $n=\dim_\bbc N\leq d$ the Euler class $\ge$ vanishes giving the isomorphism.
\end{proof}

\section{Yamazaki's Fiber Join Construction}
A procedure of Yamazaki which he called the `fiber join' allows the construction of infinitely many Sasakian structures on the total space of odd dimensional sphere bundles over a projective algebraic variety. Yamazaki \cite{Yam99} does this for K-contact structures, but as we shall see the restriction to Sasakian structures works equally well. For $j=1,\ldots,d+1$ let $\cals_j=(\xi_j,\eta_j,\Phi_j,g_j)$ be regular K-contact structures on the manifolds $M_j$ with the same smooth manifold $N$ as quotient but possibly different integral symplectic forms $\gro_j$. We assume that $d>0$ throughout. The manifold $M_j$ is the total space of a principal $S^1$ bundle over $N$. Let $L_j$ denote the complex line bundle on $N$ associated to $M_j$ such that $c_1(L_j)=[\gro_j]$. We identify $L_j$ with $M_j\times_{S^1}\bbc$. Then Yamazaki shows that the unit sphere bundle $M$ in the complex vector bundle $E=\oplus_{j=1}^{d+1}L^*_{j}$ has natural K-contact structures which we describe below. We should mention here that there is a generalization of Yamazaki's construction due to Lerman \cite{Ler04b} where the sphere bundle is replaced by a `contact fiber bundle'. This generalization was used in \cite{BGO06} (Theorem 3.5) to construct toric Sasaki manifolds that are contact fiber bundles with a toric Sasaki fiber, and a toric symplectic base. Here is Yamazaki's fiber join:

\begin{definition}\label{fiberjoindef}
The smooth manifold $M=M_1*_f\cdots *_fM_{d+1}$ defined to be the unit sphere in the complex vector bundle $E=\oplus_{j=1}^{d+1}L^*_{j}$ is called the {\bf fiber join} of the $M_j$.
\end{definition}

The line bundle $L^*_j$ has a Hermitian metric which defines a `norm' $r_j:L^*_j\rightarrow \bbr_{\geq 0}$. We
let $(r_j,\theta_j)$ denote polar coordinates on the fiber of the line bundle $L^*_j$. Then $M$ is an $S^{2d+1}$-bundle over $N$ whose fibers are defined by the equation $\sum_{j=1}^{d+1}r_j^2=1$. The fibers $S^{2d+1}$ can be thought of as the topological join 
$$S^1*\fract{d+1~times}{\cdots}* S^1=S^{2d+1}.$$

\subsection{The K-Contact Structure}
Recall Definition 6.4.7 of \cite{BG05} that a {\it K-contact structure} is a contact metric structure $(\xi,\eta,\Phi,g)$ such that $\xi$ is Killing vector field. We note that this condition is equivalent to $\pounds_\xi\Phi=0$.

The contact bundle is $\cald=\ker\eta$ and the real line bundle generated by $\xi$ is denoted by $\bbr\xi$. It is convenient to decompose the tangent bundle of $M$ into horizontal and vertical parts as $TM=\calh + \calv$. We then see that 
$\bbr\xi\subset \calv$ and $\calh\subset \cald$ and we have the decompositions
$$TM=\calh+\calv/\bbr\xi+\bbr\xi,\qquad \cald=\calh+\calv/\bbr\xi.$$

Let $\bfa=(a_1,\ldots,a_{d+1})\in(\bbr^+)^{d+1}$ and consider the 1-form on $\oplus_{j=1}^{d+1}L^*_j$ defined by\footnote{Note that $a_j=1/\grl_j$ and the two Equations \eqref{conform} and \eqref{Reebfield} have relative minus signs in Yamazaki which can be absorbed into the definition of $\theta_j$.}
\begin{equation}\label{conform}
\eta_\bfa=\sum_{j=1}^{d+1}\frac{1}{a_j}r_j^2(\eta_j+d\theta_j).
\end{equation}
We can check that this restricts to a contact form on $M$ and its Reeb vector field is
\begin{equation}\label{Reebfield}
\xi_\bfa=\frac{1}{2}\sum_{j=1}^{d+1}a_j(\xi_j+\partial_{\theta_j}).
\end{equation}
As a complex line bundle on $N$ we can identify the total space of the line bundle $L_j$ with the quotient $M_j\times_{S^1}\bbc$ where the $S^1$ action on $M_j$ is the flow of the Reeb vector field $\xi_j$ and its action on $\bbc$ is induced by the vector field $-\partial_{\theta_j}$. Thus, on $L_j$ we can identify $\partial_{\theta_j}$ with $\xi_j$ which implies that on $M$ we have the indentification
$$\xi_\bfa=\sum_{j=1}^{d+1}a_j\partial_{\theta_j}$$
in terms of the polar coordinates $\{(r_j,\theta_j)\}_{j=1}^{d+1}$ on each fiber $S^{2d+1}$ satisfying $\sum_jr_j^2=1$. 
The geometry transverse to the Reeb foliation of $\cals_\bfa=(\xi_\bfa,\eta_\bfa,\Phi_\bfa,g_\bfa)$ on the fiber join $M$ is given by the transverse symplectic form
\begin{equation}\label{transsympl}
d\eta_\bfa=\sum_{j=1}^{d+1}\frac{1}{a_j}\bigl(r_j^2d\eta_j +2r_jdr_j\wedge(\eta_j+d\theta_j)\bigr)
\end{equation}
This transverse symplectic form splits as a sum of a nondegenerate 2-form on $\calh=\sum_j\calh_j$ and a nondegenerate 2-form on $\calv=\sum_j\calv_j$ as
\begin{equation}\label{sumsymp}
d\eta_\bfa= \sum_j\frac{1}{a_j}r_j^2d\eta_j + 2\sum_j\frac{1}{a_j}(r_jdr_j\wedge (\eta_j+d\theta_j).
\end{equation} 

The transverse Riemannian metric is then given by
\begin{equation}\label{transKah}
g_\bfa^T =d\eta_a\circ(\BOne\otimes \Phi_\bfa)=\sum_{j=1}^{d+1}\frac{1}{a_j}d\bigl(r^2_j(\eta_j+d\theta_j)\bigr)\circ(\BOne\otimes \Phi_\bfa),
\end{equation}
so the K-contact metric is
\begin{equation}\label{Kconmetric}
g_\bfa= d\eta_a\circ(\BOne\otimes \Phi_\bfa) +\eta_\bfa\otimes \eta_\bfa.
\end{equation}

Theorem 3.2 of Yamazaki states that this $\cals_\bfa$ gives the fiber join $M_1*_f\cdots *_fM_{d+1}$ a K-contact structure. It will be convenient to consider the integral symplectic forms within the set of symplectic forms $\{\gro_j\}_{j=1}^{d+1}$. The following result is implicit in \cite{Yam99}

\begin{proposition}\label{imyam}
Let $M$ be the $d+1$ fiber join of Yamazaki with its induced K-contact structures. Then $M$ admits an action of the $(d+1)$-torus $T^{d+1}$ of Reeb type that leaves the K-contact structure invariant. Furthermore, the K-contact structure on $M$ restricts to the standard toric contact structure on the fiber $F_x\approx S^{2d+1}$ for each $x\in N$. 
\end{proposition}

\begin{proof}
Following \cite{Yam99} we see that the vector fields $\xi_j+\partial_{\theta_j}$ span an $(d+1)$-dimensional Abelian Lie algebra $\gt_{d+1}$ which generate the $T^{d+1}$ action. Furthermore, it is clear from Equation \eqref{Reebfield} that $\xi_\bfa\in\gt_{d+1}$, so the action is of Reeb type. Note also that the vector fields $\xi_j-\partial_{\theta_j}$ restricted to a fiber $F_x\approx S^{2d+1}$ are tangent to $F_x$. So the torus $T^{d+1}$ acts on each fiber. Thus, the Reeb vector field $\xi_\bfa$ restricted to $F_x$ is tangent to $F_x$. It follows that $\ker\eta_\bfa\cap TF_x$ is a codimension 1 distribution on $F_x$. So the restriction of $\eta_\bfa$ to $F_x$ is a contact form on $F_x$ with a $T^{d+1}$ action of Reeb type which implies that $\eta_\bfa |_{F_x}$ defines the standard contact structure on $F_x\approx S^{2d+1}$. We also note that Yamazaki shows that his K-contact structure on the sphere is contact equivalent to the weighted sphere structure of Takahashi \cite{Tak78} (This is described by Example 7.1.12 of \cite{BG05}). 
\end{proof}

\begin{remark}\label{conerem}
By defining $R=\sqrt{\sum_{j=1}^{d+1}r_j^2)}$ we can identify $\oplus_{j=1}^{d+1}L^*_j\setminus \{\underline{0}\}$ with the cone $C(M)=M\times \bbr^+$ where $\underline{0}$ denotes the zero section and $R\in\bbr^+$.
\end{remark}

\subsection{The Sasaki Case}
In this paper we are concerned with applying Yamazaki's fiber join construction to the case of Sasaki manifolds. Thus, we assume that $N$ is a smooth projective algebraic variety and that the symplectic forms are K\"ahler with respect to some complex structure\footnote{The fiber join can be generalized to the orbifold category, but we do not do so here}. So the K\"ahler classes $[\gro_j]$ are integral classes, i.e. $[\gro_j]\in H^2(N,\bbz)\cap H^{1,1}(N,\bbr)=\calk_{NS}(N)$ called the Neron-Severi lattice of the K\"ahler cone. In this case the total space $M_j$ of the principal $S^1$ bundle has a regular Sasakian structure, so the corresponding holomorphic line bundles $L_j$ are positive satisfying $c_1(L_j)=[\gro_j]$. It is important to realize that the K\"ahler forms $\gro_j$ are not necessarily distinct and on the fiber join we get deformations of Sasakian structures for each $(d+1)$-tuple $(\gro_1,\ldots,\gro_{d+1})$ of K\"ahler forms on $N$ such that $[\gro_j]\in\calk_{NS}(N)$ for all $j=1,\ldots,d+1$. 

So in the Sasaki case the fiber join construction requires a choice of $d+1$ elements (not necessarily distinct) of $\calk_{NS}(N)$. This defines the set $\gS_{d+1}$ consisting of $d+1$ elements of $\calk_{NS}(N)$, and each element of $\gS_{d+1}\otimes \bbz^+\approx (\bbz^+)^{d+1}$ gives rise to a fiber join.   For each fiber join we have a Sasaki cone's worth of isotopy classes of Sasakian (K-contact) structures. We choose an ordering on $\gS_{d+1}\otimes \bbz^+$ as follows: let $r$ denote the number of distinct elements of $\gS_{d+1}$ with $s_j$ choices of $[\gro_j]$ with $j=1,\ldots,r$. Then we have $s_1+\cdots +s_r=d+1$, and we give a partial ordering by $s_1\geq s_2\geq\cdots \geq s_r$. If $s_i=s_j$ for $i\neq j$ we make a choice to give an ordering. Within the subset consisting of elements $s_k$ copies of $[\gro_k]$ we choose the order according to $b_1\geq b_2\geq\cdots \geq b_{s_k}$ where $b_j\in\bbz^+$, and again if $b_k=b_j$ we simply make a choice. So we consider $\gS_{d+1}\otimes \bbz^+$ to be an ordered set whose elements are denoted by $\gw$. This gives an (almost) CR structure $(\cald_\gw,J)$ of Sasaki (K-contact) type. We shall denote the Sasaki (K-contact) manifold constructed by this fiber join by $M_\gw$. We denote the underlying contact bundle on $M_\gw$ by $\cald_\gw$ and write the contact manifold as $(M_\gw,\cald_\gw)$ and the corresponding CR manifold as $(M_\gw,\cald_\gw,J)$. When $M_\gw$ is understood, we denote the underlying (almost) CR structure by $(\cald_\gw,J)$.
From the fiber join construction one sees that the contact manifold $(M,\cald_\gw)$ is independent of the order of the $(M_j,\gro_j)$, so we shall simply choose a fixed order for our set $\gw$, regarding a different order as equivalent. A special case of interest is when $\gro_j=\gro$ for a fixed symplectic form $\gro$ and all $j$, that is when the integer $r=1$. In this case we write the contact structure as $\cald_\bfb$ where $\bfb=(b_1,\cdots,b_{d+1})\in\bbz^{d+1}$.

\begin{theorem}\label{Yamthm}
Let $N$ be a smooth projective algebraic variety of complex dimension $n$ with integral K\"ahler forms $\gro_j$ for $j=1,\ldots,d+1$. Let $L_j$ be positive holomorphic line bundles on $N$ defined by $c_1(L_j)=[\gro_j]$. Then the unit sphere bundle $M_\gw$ in $\oplus_jL^*_j$ with its underlying contact structure $\cald_\gw$, has a natural Sasaki CR structure $(\cald_\gw,J)$ with a $d+1$-dimensional family of Sasakian structures denoted by $\cals_\bfa=(\xi_\bfa,\eta_\bfa,\Phi_\bfa,g_\bfa)$ for each $\bfa\in(\bbr^+)^{d+1}$ such that the Sasaki automorphism group $\gA\gu\gt(\cals_\bfa)$ contains the torus $\bbt^{d+1}$.
\end{theorem}

\begin{proof}
Yamazaki shows that the structure $\cals_\bfa=(\xi_\bfa,\eta_\bfa,\Phi_\bfa,g_\bfa)$ is K-contact for each $\bfa\in(\bbr^+)^{d+1}$. So it suffices to show that the underlying almost CR structure on the contact bundle $\cald_\gw$ is integrable. For each Reeb field $\xi_\bfa\in\gt^+_{sph}$ we have the decomposition $\cald_\gw\approx \calh\oplus(\calv/\bbr\xi_\bfa)$. The isomorphism $\calh\approx\pi^*TN$ shows that the almost complex structure on $\calh$ is the lifted integrable complex structure on $TN$. Moreover, as in the proof of Proposition \ref{imyam} the almost complex structure on $\calv/\bbr\xi_\bfa$ is that of the weighted sphere. But it follows from Equation 7.1.3 of \cite{BG05} that this is independent of the weights $\bfa$ and coincides with the standard transverse complex structure on the sphere $S^{2d+1}$. Thus, for any choice of weight vector $\bfa$ the K-contact structure $\cals_\bfa$ is Sasakian. Moreover, since the Sasakian structure restricted to a fiber belongs to the standard toric contact structure of Reeb type, the Sasaki automorphism group $\gA\gu\gt(\cals_\bfa)$ contains the torus $\bbt^{d+1}$ for each Reeb vector field $\xi_\bfa$ in the Sasaki cone $\gt^+$.
\end{proof}

\begin{remark}\label{Kahrem}
Note that when the symplectic forms $\gro_j$ are K\"ahler, Yamazaki's assumption that $\sum_{j=0}^n\frac{1}{a_j}r^2_j\pi^*\gro_j$ is non-degenerate is automatically satisfied.
\end{remark}

\subsection{The Orbifold Quotients}\label{orbquot}
For each quasi-regular Reeb field $\xi_\bfa$ we have an $S^1$ orbi-bundle over the projective orbifold $\bbp_\bfa[\oplus L^*_j]$ which is an orbi-bundle over $N$ with fiber the weighted projective space $\bbc\bbp^d[\bfa]$. 
The $S^1$ action is generated by a quasi-regular Reeb vector field lying in the Sasaki subcone $\gt^+_{sph}$ of $M$. First consider the fiberwise $\bbc^*$ action with weight vector $\bfa=(a_1,\ldots,a_{d+1})\in (\bbz^+)^{d+1}$, viz. $\cala: \oplus_{j=1}^{d+1}L^*_j\ra{1.6} \oplus_{j=1}^{d+1}L^*_j$ defined by
\begin{equation}\label{C*act}
\cala(v_1,\ldots,v_{d+1})=(\grl^{a_1}v_1,\ldots,\grl^{a_{d+1}}v_{d+1}).
\end{equation}
We denote the group of this action by $\bbc^*_\bfa$. Restricting this to the circle subgroup $S^1_\bfa\subset \bbc^*_\bfa$ action on the unit sphere bundle $M$ gives the identification 
\begin{equation}\label{quotM}
(\oplus_{j=1}^{d+1}L^*_j\setminus\underline{0})/\bbc^*_\bfa=M/S^1_\bfa.
\end{equation}
Here and hereafter we assume that $\gcd(a_1,\ldots,a_{d+1})=1$.
The circle group $S^1_\bfa\subset\bbt^{d+1}$ acts only on the fibers $F_x\approx S^{2d+1}$ as a weighted circle action. It thus acts locally freely on $M_\gw$, so the quotient $M_\gw/S^1_\bfa$ is a projective algebraic orbifold. In fact, it follows from Equation \eqref{quotM} that we have the weighted projectivization:
$$M_\gw/S^1_\bfa = (\oplus_{j=1}^{d+1}L^*_j\setminus\underline{0})/\bbc^*_\bfa =: \bbp_\bfa[\oplus_{j=1}^{d+1}L^*_j]$$
which is a `fiber bundle' over $N$ whose fibers are orbifolds, specifically a weighted projective space $\bbc\bbp^d[\bfa]$ with its canonical orbifold structure. So for each weighted $\bbc^*_\bfa$ action we have the commutative diagram
\begin{equation}\label{bundiag}
\xymatrix{
S^1_\bfa \ar[d]^{id} \ar[r]  &S^{2d+1}\ar[d] \ar[r] &\bbc\bbp^d[\bfa] \ar[d]  \\
               S^1_\bfa \ar[r]&M_\gw \ar[r] \ar[d] &\bbp_\bfa(\oplus_{j=1}^{d+1}L^*_j)\ar[d]  \\
&N  \ar[r]^{id}  &N.}
\end{equation}
The regular case when $\bfa=(1,\ldots,1)$ is of particular interest to us where we have the standard projectivization \cite{BoTu82}. 
\begin{equation}\label{bundiagreg}
\xymatrix{
S^1\ar[d]^{id} \ar[r]  &S^{2d+1}\ar[d] \ar[r] &\bbc\bbp^d \ar[d]  \\
               S^1 \ar[r]&M_\gw \ar[r] \ar[d] &\bbp(\oplus_{j=1}^{d+1}L^*_j)\ar[d]  \\
&N \ar[r]^{id}  &N.}
\end{equation}

\subsection{Reduction of the Structure Group}
Since the vector bundle $E$ splits as a sum of nontrivial complex line bundles, the transition functions of $E$ takes its values in the complex torus $\bbt^{d+1}_\bbc$. They can be represented by a $d+1$ by $d+1$ diagonal matrix $\grg=(\grg_1,\ldots,\grg_{d+1})$. Now each holomorphic line bundle corresponds to an element of $H^1(N,\calo^*)$, i.e.  equivalences classes of \v{C}ech cocyles, represented by transition functions $g_{\gra\grb}\in \calo^*(U_\gra\cap U_\grb)$ where equivalence is given by 
$$g'_{\gra\grb}=\phi_\gra g_{\gra\grb}\phi^{-1}_\grb$$
for some $\phi_\gra\in\calo^*(U_\gra)$. Then the transition functions\footnote{The parentheses in the superscript indicates an independent labelling; whereas, without the parentheses such as $g^b_{\gra\grb}$ will indicate a power of $g_{\gra\grb}$.} of $\bbt^{d+1}_\bbc$ takes the form 
$$\begin{pmatrix} 
 g^{(1)}_{\gra\grb} & 0 & 0 \\
 0 & \ddots &0 \\
 0 &0 & g^{(d+1)}_{\gra\grb}
 \end{pmatrix}$$
 where $g^{(j)}_{\gra\grb}\in\calo^*(U_\gra\cap U_\grb)$. If we let $g_{\gra\grb}$ denote the transitions functions for a positive line bundle $L$, the transition functions for $L^*$ are $g^{-1}_{\gra\grb}$. The question is when can we reduce these transition functions to a $\bbc^*$ subgroup? The transition functions for such a subgroup, denoted by $\bbc^*_\bfb$, takes the form
 \begin{equation}\label{Geqn}
 G_{\gra\grb}(\bfb)=\begin{pmatrix} 
 g^{-b_1}_{\gra\grb} & 0 & 0 \\
 0 & \ddots &0 \\
 0 &0 & g^{-b_{d+1}}_{\gra\grb}
 \end{pmatrix}
 \end{equation}
where $b_j\in\bbz^+$. 
 
We make use of the following lemma whose proof is easy.  

\begin{lemma}\label{fiberjoincolinear}
Let $M_\gw=M_1*_f\cdots *_fM_{d+1}$ be a fiber join with a collection $\gS_{d+1}$ of $(d+1)$ K\"ahler classes $[\gro_j]$ (not necessarily distinct). Then the following are equivalent:
\begin{enumerate}
\item All $[\gro_j]\in\gS_{d+1}\subset \calk_{NS}(N)$ are colinear;
\item there exists a primitive K\"ahler class $[\gro_N]\in \calk_{NS}(N)$ and $b_j\in\bbz^+$ such that $[\gro_j]=b_j[\gro_N]$ for all $j=1,\ldots,d+1$;
\item after possibly tensoring $L_j$ by a flat line bundle, the line bundles $L_j$ take the form $L_j=L^{b_j}$ with $b_j\in\bbz^+$ where $c_1(L)=[\gro_N]$ for some positive holomorphic line bundle $L\in H^1(N,\calo^*)$.
\end{enumerate}
\end{lemma}

We have the exact sequence
$$0\longrightarrow H^1(N,\bbz)\longrightarrow H^1(N,\calo)\longrightarrow H^1(N,\calo^*)\fract{c_1}{\longrightarrow} H^2(N,\bbz)\longrightarrow H^2(N,\calo)$$
so generally the positive line bundles $L_j$ are not unique. Moreover, for positive line bundles $L\in H^1(N,\calo^*)$, the image $c_1(L)$ lies in the Neron-Severi lattice $\calk_{NS}(N)\subset H^2(N,\bbz)$. 

\begin{lemma}\label{C*sublem}
For a fiber join $M_\gw=M_1*_f\cdots *_fM_{d+1}$, the group $\bbt^{d+1}_\bbc$ of the bundle $E$ reduces to the $\bbc^*$ subgroup defined by $G_{\gra\grb}(\bfb)$ if and only if all $[\gro_j]\in\gS_{d+1}$ are colinear in $\calk_{NS}(N)$.
\end{lemma}

\begin{proof}
Since $M_\gw$ is a fiber join,  there exist positive line bundles $L_j$ on $N$ such that $c_1(L_j)=[\gro_j]$. Suppose further that all $[\gro_j]$ are colinear. Then by Lemma \ref{fiberjoincolinear} there exist $b_j\in\bbz$ and a positive holomorphic line bundle $L$ such that $L_j=L^{b_j}$ for all $j=1,\ldots,d+1$. So the transition functions of $L_j$ are $g^{b_j}_{\gra\grb}$ where $g_{\gra\grb}$ are the transition functions of $L$. Thus, the transition functions of $E=\oplus_jL^*_j$ reduce to the 1-parameter subgroup $G_{\gra\grb}(\bfb)$ of Equation \eqref{Geqn}.

Conversely, if we can reduce $\bbt^{d+1}_\bbc$ to the 1-parameter subgroup $G_{\gra\grb}(\bfb)$, the transition functions will take the form in Equation \eqref{Geqn} up to reordering. Since the Sasaki manifold $M_\gw$ is a fiber join, each $L_j$ is a positive line bundle. But then the holomorphic line bundle $L$ defined by the cocyle $\{g_{\gra\grb}\}$ is also positive. So there exists a primitive K\"ahler class $[\gro_N]\in\calk_{NS}(N)$ such that $c_1(L)=[\gro_N]$ and $L_j=L^{b_j}$. It follows from Lemma \ref{fiberjoincolinear} that all $[\gro_j]\in\gS_{d+1}$ are colinear.
\end{proof}

We see that the set of Sasaki fiber joins $M_\gw$ divides naturally into 2 classes: 

(1) All the K\"ahler classes $[\gro_j]\in\gS_{d+1}$ are colinear, in which case we say that the fiber join $M_\gw$ is {\it colinear}; 
\newline\indent (2) not all K\"ahler classes $[\gro_j]\in\gS_{d+1}$ are colinear, in which case we say that the fiber join is {\it non-colinear} .

We have

\begin{proposition}\label{projprod}
Consider the fiber join $M_\gw=M_1*_f\cdots *_fM_{d+1}$ defined by the collection $\gS_{d+1}$. Then 
\begin{enumerate}
\item when all K\"ahler classes $[\gro_j]\in\gS_{d+1}$ are colinear, in which case $\gro_j=b_j\gro_N$ for some positive line bundle $L$ with $c_1(L)=[\gro_N]$ and $b_j\in\bbz^+$, the orbifold quotient $\bbp_\bfa(\oplus_{j=1}^{d+1}L^*_j)$ is a product $N\times \bbc\bbp^d[\bfa]$ if and only if $\bfb=b\bfa$  where $b=\gcd(b_1,\cdots,b_{d+1}).$ Hence, the Sasaki CR structure $(\cald_\gw,J)$ is cone decomposable;
\item when not all K\"ahler classes $[\gro_j]\in\gS_{d+1}$ are colinear, the quotient orbifold $\bbp_\bfa(\oplus_{j=1}^{d+1}L^*_j)$ is not a product for any $\bfa\in (\gt^+_{sph})^*\cap\bbq$, and $(\cald_\gw,J)$ is cone indecomposable.
\end{enumerate}
\end{proposition}

\begin{proof}
Case (1): By Lemma \ref{fiberjoincolinear} we can write $L_i=L^{b_i}$ where $L$ satisfies $c_1(L)=[\gro_N]$. So we are interested in the weighted projectivization
\begin{equation}\label{proja}
\bbp_\bfa(\oplus_{j=1}^{d+1}L^*_j)=\bigl((\oplus_{j=1}^{d+1}L^*_j)\setminus\{0\}\bigr)/\bbc^*_\bfa.
\end{equation}
The $\bbc^*_\bfa$ action on $\oplus_{j=1}^{d+1}L^*_j$ is given by 
\begin{equation}\label{bfaaction}
(v_1,\ldots,v_{d+1})\mapsto(\grl^{a_1}v_1,\ldots,\grl^{a_{d+1}}v_{d+1})
\end{equation}
for some $\grl\in\bbc^*$ and $a_j\in\bbz^+$. By Lemma \ref{C*sublem} the transition functions for the vector bundle $\oplus_{j=1}^{d+1}L^*_j$ take the form of the diagonal matrix $(\grg^{b_1},\cdots,\grg^{b_{d+1}})$. Since these $\bbc^*$ actions commute, we have a global $\bbc^*_\bfa$ action on $\oplus_{j=1}^{d+1}L^*_j$, so the projectivization given by Equation \eqref{proja} is well defined.

Note that if we choose $\bfa$ such that $\bfb=b\bfa$, then we can choose the matrix of the $\bbc^*_\bfa$ action to coincide with the transition function matrix, by choosing $\grl=\grg^b$ (restricting to an open set of a good cover, for example). This reduces the induced transition functions on the projectivization $\bbp_\bfa(\oplus_{j=1}^{d+1}L^*_j)$ to the identity. Thus, it is the product
$$\bbp_\bfa(\oplus_{j=1}^{d+1}L^*_j)=N\times \bbc\bbp^d[\bfa].$$
Furthermore, for any other choice of $\bfa$ the $\bbc^*_\bfa$ action does not coincide with the transition functions, and will not reduce them to the identity.

Case (2): Since the $[\gro_j]\in\gS_{d+1}$ are not all colinear, the group of the bundle $E$ does not reduce to a $\bbc^*$ subgroup by Lemma \ref{C*sublem}, but only to a $\bbt^k_\bbc$ subgroup for $k=2,\ldots,d+1$. But then the action \eqref{bfaaction} cannot coincide with the transition functions when restricted to open sets. So the projective group does not reduce to the identity.
\end{proof}

\begin{remark}\label{Mjrem}
Note that the unit sphere bundle of $L^*_j$ is identified with $M_j$ and is the Sasaki submanifold of $M$ obtained by setting $r_i=0$ for $i\neq j$. The fact that for the Sasaki case the submanifolds $M_j$ are all Sasaki force the integers $a_j$ in the projectivization to be all positive integers.
\end{remark}

As described in \cite{BGS06} the $T^{d+1}$ action gives rise to the (unreduced) Sasaki cone $\gt^+_{sph}$ on $M$. It follows from Proposition 4.2 of \cite{Yam99} that any other Reeb vector field $\xi_{\bfa'}$ of the form \eqref{Reebfield} with $\bfa$ replaced by $\bfa'\in(\bbr^+)^{d+1}$ is an element of $\gt^+_{sph}$. It follows that 

\begin{proposition}\label{construct}
The contact structure $\cald_\bfa=\ker\eta_\bfa$ defined by Equation \eqref{conform} is independent of $\bfa$ with $\xi_\bfa\in\gt^+_{sph}$, and only depends on the regular Sasaki manifolds $M_j$, or equivalently on the collection $\gS_{d+1}$ (up to order). Accordingly, we denote the underlying contact structure by $\cald_\gw$.
\end{proposition}

\subsection{The Colinear Case}
We consider the special case when $\gro_j=b_j\gro_N$ for some K\"ahler form on $N$. Proposition \ref{projprod} says that this will be a product precisely when $\bfb=b\bfa$. We now compute the induced K\"ahler form on $N\times \bbc\bbp^d[\bfa]$.

\begin{proposition}\label{projKahform}
The K\"ahler form on $N\times \bbc\bbp^d[\bfa]$ induced by the quasiregular Sasakian structure $\cals_\bfa=(\xi_\bfa,\eta_\bfa,\Phi_\bfa,g_\bfa)$ on the fiber join $M_\gw$ is $b\gro_N+\gro_\bfa$. Equivalently, the transverse K\"ahler structure on $M_\gw$ is  
\begin{equation}\label{transsymp}
\pi^*_\bfa\bigl(b\gro_N+\gro_\bfa\bigr)=d\eta_\bfa
\end{equation}
where $\eta_\bfa$ is given in Equation \eqref{conform}. 

\end{proposition}

\begin{proof}
The K\"ahler form on the product $N\times \bbc\bbp^d[\bfa]$ takes the form $l_1\gro_N+l_2\gro_\bfa$ for some positive integers $l_1,l_2$. We have the following commutative diagram
\begin{equation}\label{cartfib}
\xymatrix{
M_i \ar[ddr]_{\pi_i} \ar[dr]^{\tpi_i} \ar[r]_\gri  & M_\gw \ar[d]^{\pi_\bfa}  \\
& N\times \bbc\bbp^d[\bfa] \ar[d]^{pr_1}  \\
& N}
\end{equation}
where $pr_1$ and the $\pi'$s are the obvious projections, and $\gri$ is the natural inclusion obtained by setting $r_j=0$ for $j\neq i$. It follows from Equation \eqref{conform} that the restriction of the Sasaki structure $\cals_\bfa=(\xi_\bfa,\eta_\bfa,\Phi_\bfa,g_\bfa)$ on $M_\gw$ to the submanifold $M_j$ is the Sasakian  structure $a_j^{-1}\cals_j$ not $\cals_j$. Indeed, we have $d\eta_\bfa|_{M_i}=a_i^{-1}d\eta_i$, so by Diagram \eqref{cartfib}
$$\tpi_i^*(l_1\gro_N+l_2\gro_\bfa)=d\eta_\bfa|_{M_i}=a_j^{-1}d\eta_i =a_i^{-1}b_i \pi_i^*\gro_N$$
which implies $l_1=\frac{b_i}{a_i}$ and $\tpi_i^*\gro_\bfa=0$. But then Proposition \ref{projprod} gives $l_1=b$. We also know that the K\"ahler form $\gro_\bfa$ on $\bbc\bbp^d[\bfa]$ is induced by the weighted Sasakian structure $\cals_\bfa$ on each fiber $S^{2d+1}$; hence, $l_2=1$.

\end{proof}

\subsection{The Non-colinear Case}
In the case where the $\gro_j$ are not generated by a single K\"ahler form $\gro_N$, there exists at least two line bundles $L_1$ and $L_2$ whose transition functions are independent, say $\grg_1$ and $\grg_2$; they generate a 2-torus $\bbt^2$ that does not reduce to a circle action.

The contact bundle is $\cald_\gw=\ker\eta_\bfa$ for all $a\in (\gt^+_{sph})^*$ and the real line bundle generated by $\xi_\bfa$ is denoted by $\bbr\xi_\bfa$. It is convenient to decompose the tangent bundle of $M$ into horizontal and vertical parts as $TM=\calh + \calv$. We then see that 
$\bbr\xi_\bfa\subset \calv$ and $\calh\subset \cald_\bfa$ and we have the decompositions
$$TM=\calh+\calv/\bbr\xi_\bfa+\bbr\xi_\bfa,\qquad \cald_\gw=\calh+\calv/\bbr\xi_\bfa.$$
In fact, Yamazaki constructs a K-contact structure $(\xi_\bfa,\eta_\bfa,\Phi_\bfa,g_\bfa)$ with integral symplectic forms $\gro_j$ satisfying $\pi_j^*\gro_j=d\eta_j$ as long as the 2-form $\sum_{j=1}^{d+1}\frac{r_j^2}{a_j}\pi^*\gro_j$ is non-degenerate. However, in the Sasaki case this non-degeneracy condition is automatic.

\subsection{Topology of the Projectivization}
For each $j$ the principal $S^1$ bundle $M_j$ over $N$ is classified by homotopy classes of maps $[N,BS^1]=[N,\bbc\bbp^\infty]$. Since the fiber join is constructed from the split complex vector bundle $E=\oplus_{j}L^*_j$, the group of the bundle reduces to the maximal complex torus $\bbt^{d+1}_\bbc=(\bbc^*)^{d+1}$ as we have previously discussed. So equivalence classes of such complex vector bundles are in one-to-one correspondence with homotopy classes of maps $[N,B(S^1)^{d+1}]=[N,(BS^1)^{d+1}]=[N,(\bbc\bbp^\infty)^{d+1}]$. The cohomology ring of the latter is the polynomial ring $\bbz[x_1,\ldots,x_{d+1}]$ where $x_j=c_1(\calo(1))$ is the universal first Chern class of the dual of the tautological bundle. Being the universal classes each $x_j$ pullsback to a negative integral K\"ahler class $-b_j[\gro_j]$.

Consider now the projectivization $\bbp(\oplus_jL^*_j)$ which is defined by projecting the bundle $E$ fiberwise. So the fibers of $\bbp(\oplus_{j}L^*_j)$ are projective spaces $\bbc\bbp^d$. That is, $\bbp(\oplus_{j}L^*_j)$ is a $\bbc\bbp^d$ bundle over $N$. The transition functions for $\bbp(\oplus_{j}L^*_j)$ are equivalences classes $\bar{\calg}$ of matrices where $\calg'\sim \calg$ if and only if there exists $\grl\in\bbc^*$ such that $\calg'=\grl\calg$. Since $c_1(L_j)=-b_j[\gro_j]$, we see that the cohomology ring is (cf. \cite{BoTu82} page 270) 
$$H^*(\bbp(\oplus_{j=1}^{d+1}L^*_j),\bbz)=H^*(N,\bbz)[x]/\Bigl(\prod_{j=1}^{d+1}(x+b_j[\gro_j])\Bigr).$$ 
Note that $x$ is a 2-class and in the denominator the coefficient of $x^{d+1-k}$ is the $kth$ elementary symmetric function $\grs_k(b_1[\gro_1],\cdots,b_{d+1}[\gro_{d+1}])$. 
So $x^{d+1}$ is a linear combination of $\{1,x,\cdots,x^d\}$ with coefficients in $H^*(N,\bbz)$ which are precisely the Chern classes of $E$. This gives
$$x^{d+1}=-\sum_{k=1}^d\grs_k(b_1[\gro_1],\cdots,b_{d+1}[\gro_{d+1}])x^{d+1-k}.$$

If we consider the unit sphere bundle $M$ in $\oplus_jL^*_j$ with its canonical contact structure $\cald_\gw$, the projectivized bundle $\bbp(\oplus_jL^*_j)$ coincides with the quotient of $M$ by the $S^1$ action generated by the regular Reeb vector field $\xi_1$ on $M$ obtained by choosing $\bfa=(1,\cdots,1)$.  Note in the special case that $c_1(L_j)=b_j[\gro_N]$ for projective bundle $\bbp(\oplus_{j=1}^{d+1}L^*_j)$ we can, without loss of generality,  order the $b_j$s as $b_1\leq b_2\leq \cdots\leq b_{d+1}$ and then take $b_1=1$ in which case $\bbp(\oplus_{j=1}^{d+1}L^*_j)$ becomes $\bbp(\BOne\oplus \oplus_{j=2}^{d+1}(L^*)^{b_j})$ where $c_1(L)=[\gro_N]$.

\subsection{The First Chern Class of $\cald_\gw$}
The first Chern class $c_1(\cald)$ of the contact bundle $\cald$ is a fundamental invariant of contact structures. For the fiber join we can express $c_1(\cald_\gw)$ in terms of invariants on $N$.

\begin{proposition}\label{c1yam}
Let $M_\gw$ be the unit sphere bundle in the vector bundle $E=\oplus_{j}L^*_j$ over $N$ obtained as the fiber join $M_\gw=M_1*_f\cdots *_f M_{d+1}$ with its canonical family of Sasakian structures with contact bundle $\cald_\gw$. Then 
$$c_1(\cald_\gw)=\pi_M^*\bigl(\sum_{j=1}^{d+1}c_1(L^*_j)+c_1(N)\bigr)=\pi_M^*\bigl(-\sum_{j=1}^{d+1}[\gro_j]+c_1(N)\bigr).$$
where $[\gro_j]$ are the K\"ahler classes in $\gS_{d+1}\subset \calk_{NS}(N)$ defined up to order.
\end{proposition}

\begin{proof}
First we set $E_0=E-\{\text{zero section}\}$ and we identify the zero section with $N$. We let $\pi_0:E_0\ra{1.5} N$ denote the restriction of $\pi:E\ra{1.5} N$ to $E_0$.  Let $h$ be a Hermitian metric on $E$. Following Milnor and Stasheff \cite{MiSt74} we construct a ``canonical'' complex $d$-plane bundle $\cale$ over $E_0$ as follows. A point of $E_0$ is specified by a fiber $F$ of $E$ together with a non-zero complex vector $v$ in $F$. Then the fiber of $\cale$ over $v\in E_0$ is the orthogonal complement of $v$ in the vector space $F$.

Now since $M$ is the unit sphere bundle in $E$ and hence in $E_0$, we can identify $E_0$ with the cone $C(M)=M\times\bbr^+$ by thinking of a point of $M$ as the pair $(x,u)$ with $x\in N$ and $||u||=1$. Then the pair $(x,v)$ with $||v||=r$ gives the triple $(x,\frac{v}{||v||},r)$ as a point of $C(M)$ with $u=\frac{v}{||v||}$. In terms of the coordinates $(r_j,\theta_j)$ on $E_0$ we identify the vector $v$ with complex coordinates $z_j=r_je^{i\theta_j}$ and $r=\sqrt{\sum_{j=1}^{d+1}r_j^2}$.

Let $\tilde{\cald}$ denote the pullback of $\cald$ to $C(M)$. We claim that under the identification of $C(M)$ and $E_0$, $\tilde{\cald}$ is identified with $\pi_0^*TN+\cale$ as real vector bundles. To see this we note that at the point $p=(x,\frac{v}{||v||},r)$ in the cone $C(M)$ the vector $\Psi_p=r(\partial_r)_p$ is perpendicular to the sphere $S^{2d+1}$ in the fiber $\pi_0^{-1}(x)\subset E_0=C(M)$. Now fix a Sasakian structure $\cals=(\xi,\eta,\Phi,g)$ on $M$ in the family with underlying contact structure $\cald$. This induces a polarized K\"ahlerian structure $(d(r^2\eta),\tilde{g}=dr^2+r^2g),I,\xi)$ on $C(M)$. Using the complex structure $I$ on $C(M)$ we obtain a complex vector $\Psi_p+I\Psi_p=\Psi_p-i\xi_p$. This gives the Reeb vector field $-\xi$ which belongs to the conjugate Sasakian structure $(-\xi,-\eta,-\Phi,g)$. This implies that the complex structure $I$ on $C(M)$ is the complex conjugate to the complex structure on $E_0$ under the identification described above. 

The orthogonal complement of $\Psi_p-i\xi_p$ in $T_pC(M)^\bbc=T_pE_0^\bbc$ with respect to the Hermitian metric $\tilde{g}$ is $\tilde{\cald}$. Thus, we can identify $\tilde{\cald}$ with $\pi_0^*TN+\cale$. This identification is independent of the choice of Sasakian structure in the family since all members correspond to the same contact  bundle $\cald$. 

With the above identification and the definition of Chern classes given in \cite{MiSt74}, we have for $n\geq 1$
$$\sum_{j=1}^{d+1}c_1(L^*_j)=c_1(\oplus_{j=1}^{d+1}L^*_j)=c_1(E)=(\pi_0^*)^{-1}c_1(\cale)=(\pi_0^*)^{-1}c_1(\tilde{\cald})-c_1(TN).$$
This implies $c_1(\tilde{\cald})=c_1(\cale)+\pi_0^*c_1(TN))$. But $c_1(L^*_j)=-[\gro_j]$.
\end{proof}

\begin{remark}\label{totchern}
This proof works for the $k$th Chern class when $2k< 2d+1$, since then $\pi^*_0:H^{2k}(N,\bbz)\longrightarrow H^{2k}(E_0,\bbz)$ is an isomorphism. This gives the $k$th Chern class in terms of the $k$th elementary symmetric polynomial, viz
\begin{equation}\label{ckeqn}
c_k(\cald)=\pi^*_M\bigl(\grs_k(-[\gro_1],\ldots,-[\gro_{d+1}]) +c_k(N)\bigr).
\end{equation}
If $d_0=d_\infty=0$ it does not hold for $c_2$ in which case $c_2$ is the Euler class of the real bundle $E_\bbr$.
\end{remark}

We can rewrite the equation of Proposition \ref{c1yam} in terms of the ordering described previously. 
\begin{corollary}\label{c1cor}
For the general fiber join $M_\gw$ we have
\begin{enumerate}
\item $\pi_M^*:H^2(N,G)\ra{1.6} H^2(M_\gw,G)$ is an isomorphism for any coefficients $G$.
\item $c_1(\cald_\gw)=$
$$\pi_M^*\Bigl(-\bigl(\grs_1(b_1,\ldots,b_{s_1})[\gro_1]+\cdots +\grs_1(b_{s_{r-1}+1},\ldots,b_{s_{r-1}+s_r})[\gro_r]\bigr) +c_1(N)\Bigr),$$
where $\grs_1$ is the first elementary symmetric function.
\item $w_2(M_\gw)=\pi^*_M\bigl(w_2(E)+w_2(N)\bigr).$
\item $M_\gw$ is a spin manifold if and only if $E$ and $N$ are either both spin or both non-spin.
\end{enumerate}
\end{corollary}

\begin{proof}
Since $d\geq 1$ item (1) follows from the Gysin sequence while items (2),(3), and (4) are immediate.
\end{proof} 

In the colinear case it follows directly that 
\begin{equation}\label{colinc1}
c_1(\cald_\bfb)=\pi_M^*\bigl(-|\bfb|[\gro_N]+c_1(N)\bigr)
\end{equation}

\subsection{Relation between Fiber Joins and Regular Sphere Joins}
In this section we prove that a colinear fiber join is equivalent to a regular sphere join.
We refer to the join $M\star_{l_1,1}S^{2d+1}_{\bfw}$ as studied in \cite{BoTo14a}, but here with a $2d+1$ dimensional sphere, as a {\it regular sphere join}.

\begin{proposition}\label{joinsident}
Let $M$ be the Sasaki manifold which as a principal $S^1$ bundle over $N$ has Euler class $[\gro_N]$ and let 
$$M_{\bfb}=M_{1}*_f\fract{k~times}{\cdots}*_fM_{d+1}$$ 
be a colinear fiber join where $M_{j}=M/\bbz_{b_j}$ has Euler class $b_j[\gro_N]$ for $j=1,\dots,d+1$. Then the fiber join $M_\bfb$ can be identified with the regular sphere join $M\star_{l,1}S^{2d+1}_{\bfw}$, where $\bfb=b\bfw$ and $l=b=\gcd(b_1,\ldots,b_{d+1})$.
\end{proposition}

\begin{proof}
Since the fiber join $M_\bfb$ is colinear, it has a product projectivization $N\times \bbc\bbp^d(\bfw)$ of the Reeb vector field defined by $\xi_\bfw$ precisely when $\bfb=b\bfw$ by (1) of Proposition \ref{projprod}, and by Proposition \ref{projKahform} the K\"ahler form on $N\times \bbc\bbp^d(\bfw)$ is $b\gro_N+\gro_\bfw$.

Now one easily sees from \cite{BoTo14a} that the regular $S^{2d+1}_\bfw$ join also has the projectivization $N\times \bbc\bbp^d(\bfw)$ with K\"ahler form $l\gro_N+\gro_\bfw$. Thus, to identify the principal $S^1$ bundles we must take $l=b$. This identifies the transverse K\"ahler structures of the fiber join $M_\bfb$ and the regular sphere join $M\star_{l,1}S^{2d+1}_{\bfw}$. However, this will not uniquely determine a Sasakian structure \cite{Noz14} up to isomorphism unless the first Betti number $b_1(M)=0$. Nevertheless, we can always choose the connection 1-forms to coincide which will identify the joins.
\end{proof}

Recall \cite{BHLT16} that a Sasaki CR structure $(\cald,J)$ is {\it cone decomposable} if there is a decomposable Sasakian structure $\cals_\bfa\in \gt^+$. This implies by construction that $\cals_\bfa$ is quasiregular and its quotient orbifold is a product. Hence, by (2) of Proposition \ref{projprod} a non-colinear fiber join is cone indecomposable. Thus, we have
\begin{corollary}\label{fiberjoindecomp}
A fiber join $M_\gw$ is cone decomposable if and only if it is colinear.
\end{corollary}

We can now rephrase Proposition \ref{projprod} as
\begin{proposition}\label{projprod2}
A fiber join $M_\gw$ has a quasiregular Reeb vector field $\xi_\bfa\in\gt^+_{sph}$ whose quotient is the product orbifold $N\times \bbc\bbp^1[\bfa]$ if and only if it is cone decomposable. Furthermore, if the Picard number of $N$ is $1$ any fiber join $M_\gw$ over $N$ is cone decomposable.
\end{proposition}

\begin{remark}\label{convrem}
Proposition \ref{joinsident} says that a colinear fiber join is equivalent to a regular fiber join of the form $M\star_{l_1,1} S^{2d+1}_\bfw$; nevertheless, there are sphere bundles that are regular joins with $l_2>1$. So sphere bundles that are regular joins  are more general than colinear fiber joins. Here is a family of counterexamples giving regular joins that are not fiber joins.

\begin{example}\label{Ypqex}
The well known $Y^{p,q}$ first studied in \cite{GMSW04a} and described as a regular $S^3_\bfw$ join in \cite{BoTo14a} are such counterexamples for all $p>1$. Indeed Example 6.8 of \cite{BoTo14a} shows that 
$$Y^{p,q}= S^3\star_{l_1,l_2}S^3_\bfw$$
where
$$ l_1=\gcd(p+q,p-q),\qquad l_2=p,\qquad \bfw=\frac{1}{l_1}(p+q,p-q),$$ 
and it is diffeomorphic to $S^2\times S^3$ for $1\leq q\leq p$ with $\gcd(p,q)=1$. Since $l_2=p>1$, $Y^{p,q}$ is a regular $S^3$ join but not a fiber join.
\end{example}
\end{remark}

\section{Extremal, CSC, and Einstein Sasaki Metrics on Fiber Joins}
In this section we will apply Proposition \ref{joinsident} and known existence results to explore examples of fiber joins that admit Sasaki structures that are extremal, CSC, or even Sasaki Einstein (SE). We begin by describing important admissible conditions from \cite{ACGT08}.

\subsection{The Admissible Conditions and proof ot Theorem \ref{gennoncothm}}\label{admsect}
For a fiber join as defined in Theorem \ref{Yamthm} we have that the complex manifold arising as the quotient of the regular Reeb vector field $\xi_1$ 
is equal to $\bbp\left(\oplus_{j} L^*_j\right) \rightarrow N$. Thus, in some cases this will be an {\em admissible} projective bundle as defined in 
\cite{ACGT08}. Specifically, following Section 1.2 of \cite{ACGT08} this happens exactly when the following all hold true:
\begin{enumerate}
\item The base $N$ is a local product of K\"ahler manifolds $(N_a,\Omega_a)$, $a \in \cala \subset \bbn$, where $\cala$ is a finite index set.
\item There exist $d_0, d_\infty \in \bbn \cup \{0\}$, with $d=d_0+d_\infty +1$, such that $E_0:= \oplus_{j=1}^{d_0+1} L^*_j$ and
$E_\infty := \oplus_{j=d_0+2}^{d_0+d_\infty +2} L^*_j$ are both projectively flat hermitian holomorphic vector bundles.
This would, for example, be true if $L^*_j= L_0$ for $j=1,...,d_0+1$ and $L^*_j= L_\infty$ for $j=d_0+2,...,d_0+d_\infty+2$, where $L_0$ and $L_\infty$ are some holomophic line bundles. That is, $E_0=L_0\otimes \bbc^{d_0+1}$ and $E_\infty = L_\infty\otimes\bbc^{d_\infty+1}$. 
More generally,
$c_1(L^*_1) =  \cdots = c_1(L^*_{d_0+1})$ and $c_1(L^*_{d_0+2})= \cdots = c_1(L^*_{d_0+d_\infty+2})$
would be sufficient.
\item $\frac{c_1(E_\infty)}{d_\infty+1}-\frac{c_1(E_0)}{d_0+1}= \sum_{a\in \cala}  [\epsilon_a\Omega_a]$, where $\epsilon_a=\pm 1$.
\end{enumerate}
The K\"ahler cone of the total space of an admissible bundle $\bbp\left( E_0\oplus E_\infty\right) \rightarrow N$ has a subcone of so-called {\em admissible K\"ahler classes} (defined in Section 1.3 of \cite{ACGT08}). This subcone has dimension $|\cala|+1$ and, in general, this is not the entire K\"ahler cone. However, by Remark 2 in \cite{ACGT08}, if $b_2(N_a)=1$ for all $a\in \cala$ and $b_1(N_a) \neq 0$ for at most one $a\in \cala$, then the entire K\"ahler cone is indeed admissible.
\begin{definition}
Any fiber join $M_\gw$ where the quotient of the regular Reeb vector field $\xi_1$ 
is admissible will also be called {\bf admissible}.
If further every K\"ahler class in the K\"ahler cone of this quotient is admissible, then we call $M_\gw$ {\bf super admissible}.
\end{definition}

We now recall Proposition 11 of \cite{ACGT08}, which in turn is just a slight generalization of the the works by Guan \cite{Gua95}, Hwang \cite{Hwa94}, and Hwang-Singer \cite{HwaSi02}.

\begin{proposition} \cite{ACGT08} \label{acgtpr11}
Suppose that $S=\bbp(E_0 \oplus E_\infty) \rightarrow N$ is admissible where
$N$ is a local K\"ahler product of non-negative CSC metrics. Then every admissible
K\"ahler class contains an (admissible) extremal K\"ahler metric.
\end{proposition}

In Appendix \ref{APB} we shall describe what we mean by {\em admissible metrics}.
As a corollary to Proposition \ref{acgtpr11} we then have the following result which then proves Theorem \ref{gennoncothm}.

\begin{proposition}\label{gennonco}
Suppose $M_\gw$ is a super admissible fiber join where the local K\"ahler product $N$ is a product of K\"ahler metrics of non-negative constant scalar curvature. Then the Sasakian structure defined by the regular Reeb vector field $\xi_1$ has an extremal representative which is contained in an open set of extremal Sasakian structures.
\end{proposition}

\begin{remark}
In \cite{ACGT08} and related papers, more specific results occur regarding the existence of smooth extremal K\"ahler metrics, CSC K\"ahler metrics and KE metrics on the total spaces of admissible projective bundles. In general, these existence results - especially for the CSC and KE cases - depend on the K\"ahler classes\footnote{Proposition \ref{gennonco} above and Proposition \ref{gencoprop} below exhibit applications of exceptions to this rule.}. Thus, it is not feasible to apply these results to the regular quotient of admissible fiber joins unless we identify the K\"ahler class of the quotient K\"ahler structure. Excluding the trivial case where $c_1(L_1)=\cdots = c_1(L_d)$ (and thus the regular quotient is a (local) product), in this paper we will only explore this in the admissible case with $d=1$.
\end{remark}

In the case $d=1$ Proposition \ref{joinsident} tells us that the colinear fiber join is identical to the regular sphere join $M\star_{l,1}S^{3}_{\bfw}$, as explored thoroughly in \cite{BoTo14a}, where $\bfb=b\bfw$ and $l=b=\gcd(b_1,b_{2})$ with $b_1\geq b_2$. This allows us to apply Theorems 1.1, 1.2, and 1.4 of \cite{BoTo14a} as follows:
\begin{theorem}\label{JGAthm}
Let $M_\gw$ be the Sasaki manifold which as a principal $S^1$ bundle over $N$ has primitive Euler class $[\gro_N]$ and let 
$$M_\gw=M_{\bfb}=M_{1}*_fM_{2}$$ 
be a cone decomposable fiber join where $M_{j}=M/\bbz_{b_j}$ has Euler class $b_j[\gro_N]$ for $j=1,2$.  Assume that the K\"ahler structure of $(N,\omega_N)$ has constant scalar curvature. Then
\begin{enumerate}
\item there exist a Reeb vector field in the Sasaki subcone $\gt^+_{sph}(\cald_\bfb,J)$ of $M_{\bfb}$ such that the corresponding ray of Sasaki structures has constant scalar curvature;
\item if the scalar curvature of $\omega_N$ is non-negative, $\gt^+_{sph}(\cald_\bfb,J)$  is exhausted by extremal Sasaki metrics;
\item assuming further that $\gro_N$ is positive K\"ahler-Einstein, there exists a Reeb vector field in the Sasaki subcone $\gt^+_{sph}(\cald_\bfb,J)$ of $M_\bfb$ such that the corresponding Sasaki metric is Sasaki-Einstein if and only if $b_1+b_2=\cali_N$. Moreover, up to equivalence the number of such SE metrics equals the number of partitions of $\cali_N$ into the unordered sum of two positive integers.
\end{enumerate}
\end{theorem}

For $d=1$, admissible, and non-colinear we work out the K\"ahler class of the regular quotient in Appendix \ref{kahclass}. We will then apply this to the special case explored in Section \ref{rsprod}.

As mentioned above in the case $d>1$ known existence results here are somewhat sporadic; nevertheless we do have

\begin{proposition}\label{d>1extremal}
Let $M$ be the Sasaki manifold which as a principal $S^1$ bundle over $N$ has primitive Euler class $[\gro_N]$ and let 
$$M_{\bfb}=M_{1}*_f\cdots *_fM_{d+1}$$ 
be a cone decomposable fiber join where $M_{j}=M/\bbz_{b_j}$ has Euler class $b_j[\gro_N]$ for $j=1,\dots,d+1$. Assume that the K\"ahler structure of $(N,\omega_N)$ is extremal. Then the Sasaki structure corresponding to  the quasi-regular Reeb field $\xi_\bfa$ with $\bfa$ such that
$b\bfa=\bfb$, where $b=\gcd(b_1,\dots,b_n)$ is extremal 
(up to isotopy). Thus, as long as $(N,\omega_N)$ is extremal, both the Sasaki cone $\gt^+(\cald_\bfb,J)$ and its subcone $\gt^+_{sph}(\cald_\bfb,J)$ of $M_{\bfb}$ will always contain an open set of Sasaki extremal structures.
\end{proposition}
\begin{proof}
This result follows directly from Proposition \ref{projprod}, Proposition \ref{projKahform}, the fact that $\bbc\bbp_{\bfa}$ admits a canonical extremal K\"ahler metric, and the fact that a product of two extremal K\"ahler metrics is again extremal. Openness holds by \cite{BGS06}.
\end{proof}

\begin{remark}\label{notadmis}
Notice that we do not assume that $M_\bfb$ is admissible in Proposition \ref{d>1extremal}.
\end{remark}

\subsection{Sasaki-Einstein Metrics on Fiber Joins}
We briefly consider the Fano case with Fano index $\cali_N$. Since here $N$ is a smooth projective Fano variety of complex dimension $n$, we recall the well known result of Koboyashi and Ochia \cite{KoOc73} that $\cali_N\leq n+1$ and equality holds if and only if $N=\bbc\bbp^n$. Moreover, when $\cali_N=n$ the variety $N$ is a quadric. So the case when $N$ is smooth is quite restrictive. This is not at all so restrictive when $N$ is an orbifold \cite{BoTo18c,BoTo19a}. However, here we have the following obstructions

\begin{proposition}\label{SEprop}
Let $M_\gw=M_1\star_f\cdots \star_fM_{d+1}$ be a Sasaki fiber join with $\gS_{d+1}=\{[\gro_j]\}_{j=1}^{d+1}$, and suppose $M_\gw$ admits an SE metric. 
\begin{enumerate}
\item Then $c_1(N)=\sum_{j=1}^{d+1}[\gro_j]$. 
\item If $M_\gw$ is cone decomposable with $\gro_j=b_j\gro_N$ for some primitive K\"ahler form $\gro_N$, which is equivalent to the regular join $M\star_{l,1}S^{2d+1}_\bfw$, then
\begin{enumerate}
\item $n\geq d$ and $n=d$ if and only if $\bfw=(1,\fract{d+1}{\ldots},1)$,
\item  $\cali_N=l|\bfw|$ which implies $\cali_N\geq l(d+1)\geq 2$ and that $|\bfw|$ must divide $\cali_N$.
\item If $M_\gw$ is also admissible then 
$$\cali_N=(d_0+1)b_0+(d_\infty +1)b_\infty\leq n+1.$$
\end{enumerate}
\item If $M_\gw$ is cone indecomposable and admissible, then 
\begin{equation}\label{c10cond}
c_1(N)=(d_0+1)[\gro_0]+(d_\infty +1)[\gro_\infty].
\end{equation} 
\end{enumerate}
In particular, if $d_\infty=d_0\geq n$ there are no SE metrics for any $N$.
\end{proposition}

\begin{proof}
A necessary condition to have a Sasaki-Einstein (SE) metric is that the real first Chern class of the contact bundle must vanish (we will ignore any torsion). Then the first statement follows immediately from Proposition \ref{c1yam} from which it also follows that $N$ must be Fano. Applying this to the colinear fiber join and using Proposition \ref{joinsident}, we have from \eqref{colinc1} that $\cali_N=|\bfb|=l|\bfw|.$ Thus, we have
\begin{equation}\label{Fanoob}
d+1\leq l(d+1)\leq l|\bfw|=|\bfb|=\cali_N\leq n+1
\end{equation}
which implies both (a) and (b). To prove (c) we note that the admissible conditions imply \eqref{c10cond}, so setting 
$$\gro_0=b_0\gro_N, \qquad \gro_\infty=b_\infty\gro_N$$
for some $b_0,b_\infty\in\bbz^+$ and some primitive K\"ahler form $\gro_N$ gives the result.
The proof of the first statement in (3)  is straightforward. The last statement then follows from \cite{KoOc73}.
\end{proof}

\begin{remark}\label{ind1rem}
In the colinear case $M_\bfb$ it follows from (2) of Proposition \ref{SEprop} that the subcone $\gt^+_{sph}$ contains no SE metric when the Fano index $\cali_N$ is $1$.
\end{remark}

\section{Applications}
We now apply the results of the previous sections to describe some explicit examples where the hypotheses of Propositions \ref{gencoprop} and \ref{gennonco} hold.

\subsection{$N=\bbc\bbp^n$}
Any fiber join $M_\bfb=M_{b_1}\star_f\cdots \star_f M_{b_{d+1}}$ over $\bbc\bbp^n$ is cone decomposable and is determined by a vector $\bfb$ with coefficients $b_j\in \bbz^+$ such that $M_j$ has Euler class $b_j[\gro_{FS}]$ and is the lens space $S^{2n+1}/\bbz_{b_j}$. The corresponding Sasaki CR structure on $M_\bfb$ is denoted by $(\cald_\bfb,J)$. Now Sasaki contact structures can be distinguished by their first Chern class for which we have
\begin{equation}\label{c1cpn}
c_1(\cald_\bfb)=(n+1-|\bfb|)\pi^*_M[\gro_{FS}].
\end{equation}
Moreover, from Proposition \ref{joinsident} it is equivalent to the regular join $S^{2n+1}\star_{b,1}S^{2d+1}_\bfw$ with $\bfb=b\bfw$ and $b=\gcd(b_1,\ldots,b_{d+1})$, and for $d=1$ the much stronger Theorem \ref{JGAthm} applies giving constant scalar curvature Sasaki metrics in $\gt^+_{sph}(\cald_\bfb,J)$ as done in \cite{BoTo14a}. When $d>1$ the best we can currently do comes from Proposition \ref{d>1extremal}. In any case we also have

\begin{proposition}\label{cpnprop}
Let $\{M_\bfb\}$ be a collection of fiber joins over $\bbc\bbp^n$ with fixed integral cohomology ring and fixed Pontrjagin classes. Then for each associated Sasaki CR structure $(\cald_\bfb,J)$ its Sasaki cone and as well as the subcone $\gt^+_{sph}(\cald_\bfb,J)$ contains an open set of extremal Sasaki metrics. If we assume also that $n\leq d$, then there are a finite number of diffeomorphism types within the collection $\{M_\bfb\}$. Hence, for some diffeomorphism type there exists a countable infinity of inequivalent Sasaki contact structures all of which have an open set of extremal Sasaki metrics.
\end{proposition}

\begin{proof} 
The first statement follows directly from Proposition \ref{d>1extremal}. The second statement follows since when $n\leq d$ Lemma \ref{forlem} implies that the minimal model \eqref{rathomsphbun} is formal, so the hypothesis of Theorem 13.1 of \cite{Sul77} is satisfied. The last statement then follows immediately from Equation \eqref{c1cpn}.
\end{proof}

\begin{remark}\label{n>drem}
On the other hand if $n>d$ the Euler class $\ge$ of the sphere bundle cannot vanish, otherwise a single odd class of degree less than one half the dimension would survive to $E_\infty$ in the Leray-Serre spectral sequence, so by Leray's Theorem $M$ could not have a Sasakian structure which it does by Theorem \ref{Yamthm}. The Euler class produces torsion $\bbz_e$ in $H^{2d+2}(M,\bbz)$ where $\ge=e[\gro_{FS}]^{d+1}$ when $e>1$.
\end{remark}

A particular case of interest is the cohomological Einstein case $c_1(\cald_\bfb)=0$.
Using (1) of Corollary \ref{c1cor} we have solutions to $c_1(\cald_\bfb)=0$ if and only if 
\begin{equation}\label{c1=0cond}
n+1=|\bfb|=\sum_{j=1}^{d+1}b_j.
\end{equation}
As stated in Proposition \ref{SEprop} there are solutions if and only if $n\geq d$ and if $n=d$ there is a unique solution, namely $b_j=1$ for all $j=1,\ldots,n+1$. Generally the number of solutions is given by the number of partitions of $n+1$ as the unordered sum of $d+1$ positive integers. 

As a specific example of (3) of Theorem \ref{JGAthm} we have 

\begin{corollary}\label{41thm}
Let $(\cald_\bfb,J)$ be a Sasaki CR structure of a rank 2 colinear fiber join on an $S^3$ bundle over $\bbc\bbp^n$. Then $\gt^+_{sph}$ admits an SE metric if and only if $b_1+b_2=n+1$. Moreover, up to equivalence the number of such Sasaki CR structures equals the number of partitions of $n+1$ into the unordered sum of two positive integers.
\end{corollary}

For $N=\bbc\bbp^1$ and $d=2$, we can apply a recent result by Legendre \cite{Leg18} and state the following result about the regular ray.
\begin{proposition}\label{LegSasprop}
Let $M$ be the Sasaki manifold which as a principal $S^1$ bundle over $\bbc\bbp^1$ has primitive Euler class $[\gro_{FS}]$ and let 
$$M_{\bfb}=M_{1}*_f M_{2}*_f M_{3}$$ 
be a (necessarily cone decomposable) fiber join where $M_{j}=S^3/\bbz_{b_j}$ has Euler class $b_j[\gro_{FS}]$ for $j=1,2,3$ where $\gro_{FS}$ denotes the standard Fubini-Study K\"ahler form. Then the Sasaki structure corresponding to  the regular Reeb field $\xi_1$ is extremal 
(up to isotopy), so both $\gt^+$ and its subcone $\gt^+_{sph}$ have open sets of extremal Sasaki metrics.
\end{proposition}
\begin{proof}
Since the regular quotient equals $\bbp(\calo(-b_1)\oplus\calo(-b_2)\oplus\calo(-b_3)) \rightarrow \bbc\bbp^1$ this follows directly from Corollary 1.6 in \cite{Leg18}. The key is that Legendre proves existence of extremal K\"ahler metrics in {\em every} K\"ahler class of 
$\bbp(\calo(-b_1)\oplus\calo(-b_2)\oplus\calo(-b_3)) \rightarrow \bbc\bbp^1$ and thus we have avoided the shortcomings of not knowing 
which of those classes is the K\"ahler class of the transverse K\"ahler metric.
\end{proof}

\begin{remark}\label{Legrem}
It is interesting that the proof of Proposition \ref{LegSasprop} does not use the admissible construction. In fact, the projective bundle associated to $M_\bfb$ is not admissible in the sense of \cite{ACGT08}.
\end{remark}

\begin{remark}\label{Grothrem}
By a famous theorem of Grothendieck every complex vector bundle on $\bbc\bbp^1$ splits as a sum of line bundles. So when $N=\bbc\bbp^1$ every sphere bundle with structure group $SO(d+1)$ reduces to a $(d+1)$-dimensional torus $\bbt^{d+1}$ and  admits a fiber join construction.
\end{remark}

\subsection{$N$ is a Riemann surface $\Sigma_g$ of genus $g$}
In this case all Sasaki CR structures are cone decomposable, and there are precisely two diffeomorphism types, the trivial bundle $\grS_g\times S^{2d+1}$, and the non-trivial bundle $\grS_g\tilde{\times} S^{2d+1}$ which are distinguished by their second Stiefel-Whitney class as discussed by Yamazaki. Here we consider those related to 
admissible projective bundles as described in Section \ref{admsect}, namely 
\begin{equation}\label{vbeq}
\oplus_{j=1}^{d+1}L^*_j=E_0\oplus E_\infty=L_0\otimes \bbc^{d_0+1}\oplus L_\infty\otimes\bbc^{d_\infty+1}
\end{equation}
where $L_0,L_\infty$ are line bundles that satisfy $c_1(L_0)=-b_1 [\Omega_{\Sigma_g}]$ and $c_1(L_\infty)=-b_2[\Omega_{\Sigma_g}]$ where 
$\Omega_{\Sigma_g}$  is the standard area form on $\grS_g$, which has constant scalar curvature. Then $d=d_0+d_\infty+1$, $L_0=L^*_j$ for $j=1,\dots,d_0+1$, and $L_\infty=L^*_j$ for $j=d_0+2,\dots,d_0+d_\infty+2$. Note that 
$$\frac{c_1(E_\infty)}{(d_\infty+1)}-\frac{c_1(E_0)}{(d_0+1)}=(b_1-b_2)[\Omega_{\Sigma_g}].$$
The quotient of the regular Reeb vector field equals 
$\bbp(E_0\oplus E_\infty) \rightarrow \Sigma_g $. Since the fiber join in this case is super admissible, we can apply Theorem 6 of \cite{ACGT08} to arrive at the following.
\begin{proposition}\label{gencoprop}
Let $M$ be the Sasaki manifold which as a principal $S^1$ bundle over a compact Riemann surface $\Sigma_g$ has primitive Euler class $[\grO_{\Sigma_g}]$ and let 
$$M_{\bfb}=M_{1}*_f\cdots*_fM_{d+1}$$ 
be a (necessarily cone decomposable) fiber join where $M_{j}=M/\bbz_{b_j}$ has Euler class $b_1[\grO_{\Sigma_g}]$ for $j=1,\dots,d_0+1$ and Euler class $b_2[\grO_{\Sigma_g}]$ for $j=d_0+2,\dots,d_0+ d_\infty+2$, where $d=d_0+d_\infty+1$. 
\begin{enumerate} 
\item For genus $g$ of $\Sigma_g$ equal to $0$ or $1$, the Sasaki structure on $M_\bfb$  corresponding to  the regular Reeb field $\xi_1$ is extremal 
(up to isotopy). 
\item For genus $g>1$ and $-d_0(d_0+1) \leq \frac{2(1-g)}{b_1-b_2} \leq d_\infty(d_\infty+1)$, the Sasaki structure on $M_\bfb$  corresponding to  the regular Reeb field $\xi_1$ is extremal 
(up to isotopy). 
\end{enumerate}
In both cases there is an open set of extremal Sasaki metrics in both the Sasaki cone $\gt^+$ and its subcone $\gt^+_{sph}$.
\end{proposition}

The case $d=1 (d_0,d_\infty=0)$ was treated in detail in \cite{BoTo11,BoTo13}.

\subsection{$N$ is a product of Riemann Surfaces}\label{rsprod}
Consider $N=\grS_{g_1}\times \grS_{g_2}$ where $\grS_g$ is a Riemann surface of genus $g$, and let $M_\gw$ be a Sasaki fiber join on an $S^{2d+1}$ bundle over $\grS_{g_1}\times \grS_{g_2}$. 
We choose a complex structure on $\grS_g$ and as before let $\grO_i$ denote the standard area form on $\grS_{g_i}$.
With slight abuse of notation, we denote the pull-back of their K\"ahler classes to $H^2(N,\bbz)$ by $[\grO_1]$ and $[\grO_2]$.
The K\"ahler cone of $N$ then equals $span_{\bbr^+}\{[\grO_1],[\grO_2]\}$. The fiber join $M_\gw$ is formed from the set $\gS_{d+1}$ of K\"ahler classes which are judiciously represented by K\"ahler forms 
\begin{equation}\label{kmatrixeqn}
\gro_j=k^1_j\grO_1 +k^2_j\grO_2, \qquad k^1_j,k^2_j\in \bbz^+.
\end{equation}
So the choices of K\"ahler forms is given by the $2$ by $d+1$ matrix
\begin{equation}\label{Kmatrix}
K=
\begin{pmatrix} 
k^1_1 & k^2_1 \\
\vdots & \vdots \\
k^1_{d+1} & k^2_{d+1}.
\end{pmatrix}
\end{equation}
with values in $\bbz^+$. From the choice of $d+1=d_0+d_\infty +2$ K\"ahler forms, at most two are linearly independent which we denote by $\gro_0,\gro_\infty$ and which we allow to be linearly independent. In order to apply the admissible conditions we assume that $E$ takes the form of Equation \eqref{vbeq}. In this case we take  
\begin{equation}\label{kmatrixeqn2}
\gro_0=k^1_0\grO_1 +k^2_0\grO_2, \qquad \gro_\infty =k^1_\infty\grO_1 +k^2_\infty \grO_2,
\end{equation}
so $c_1(L_0)=-[\gro_0]$ and $c_1(L_\infty)=-[\gro_\infty]$.
We thus arrive at the K\"ahler forms
\begin{equation}\label{d0dinftyeqns}
\gro_{j}=\begin{cases} \gro_0 & ~\text{for $j=1,\ldots,d_0+1$}, \\
             \gro_\infty & ~\text{for $j=d_0+2,\ldots,d_0+d_\infty +2$}.
              \end{cases}
\end{equation}
The fiber join $M_\gw$ with $\gS_2=\{[\gro_0],[\gro_\infty]\}$ gives rise to an infinite family of inequivalent Sasaki CR structures $(\cald_K,J)$ parameterized by the matrix with some obvious relations
\begin{equation}\label{22matrix}
K= \begin{pmatrix}
     k^1_0 & k^2_0 \\
k^1_{\infty} & k^2_{\infty}
\end{pmatrix}, \qquad k^j_i\in \bbz^+.
\end{equation}
The fiber join $M_\gw$ is colinear if and only if  $\det K=0$.
Since there is nothing special about the order of $\gro_0$ and $\gro_\infty$ we have equivalence under the interchange of the rows of $K$. Furthermore, interchanging the $\bbc\bbp^1$'s corresponds to interchanging the columns of $K$. Thus, we are interested in the set $\gK(M_\gw)$ of Sasaki CR structures defined by the equivalence classes of $2$ by $2$ matrices over $\bbz^+$ where matrices $K,K'$ are equivalent if they are equivalent under interchange of rows or interchange of columns. Note that under this equivalence $|{\rm det} K|=|{\rm det} K'|$, but this relation does not determine the equivalence class. However, it is clear from this relation that the cardinality of the set $\gK(M_\gw)$ is $\aleph_0$. In what follows we choose a representative $M_K$ in $\gK(M_\gw)$ and think of $M_K$  as a manifold with its underlying Sasaki CR structure $(\cald_K,J)$ and its family of Sasakian structures parameterized by $\gt^+(\cald_K,J)$.

\subsubsection{Topological analysis}
This splits naturally into two cases, $d=1$ and $d>1$. 

\begin{proposition}\label{k>1topprop}
Let $M_K$ be a Sasaki fiber join on an $S^{2d+1}$ bundle over $\grS_{g_1}\times \grS_{g_2}$. Then
\begin{enumerate}
\item  when $d>1$ $M_\gw$  has the integral cohomology groups of the product $\grS_{g_1}\times \grS_{g_2}\times S^{2d+1}$,
\item when $d=1$  the integral cohomology groups are
\begin{equation}\label{prodRcoh}
H^p(M_K,\bbz)=\begin{cases} \bbz & \text{if $p=0,7$} \\
                                                 \bbz^{2g_1+2g_2} & \text{if $p=1,3,6$} \\
                                                 \bbz^{4g_1g_2+2} & \text{if $p=2,5$} \\
                       \bbz^{2g_1+2g_2} \oplus \bbz_{e} & \text{if $p=4$} \\
                                                                                         0 &\text{otherwise}
                                                                                         \end{cases}  
\end{equation}
where $e=k^1_0k^2_\infty+k^2_0k^1_\infty$.
\end{enumerate}
\end{proposition}

\begin{proof}
In case (1) the Euler class of the bundle vanishes, so the spectral sequence collapses. In case (2) the $E_2$ term of the spectral sequnce is
$$E_2^{p,q}=H^p(\calu,\calh^q(S^3))=\begin{cases} H^p(\calu) & \text{if $q=0,3$} \\
                                                                                         0 &\text{otherwise}
                                                                                         \end{cases}  $$
where $\calu=\bigcup_\gra U_\gra$ is a good cover of $\grS_{g_1}\times\grS_{g_2}$ and $\calh^q$ is the derived functor sheaf $U_\gra\mapsto H^q(\pi^{-1}(U_\gra))$. Let $u$ be the top class on the fiber. Then since $M_\gw$ has a Sasakian structure $d_4(u)$ which is the Euler class of the bundle cannot vanish. Using \eqref{kmatrixeqn2} we compute the Euler class
$$\ge=c_1(L_0)\cup c_1(L_\infty)=(-[\gro_0])(-[\gro_\infty])=[\gro_0][\gro_\infty]=(k^1_0k^2_\infty+k^2_0k^1_\infty)[\grO_1][\grO_2].$$
Then since $d_4(u)=\ge$ is the only non-vanishing differential in the spectral sequence we get the result.



\end{proof}

There are a countable infinity of distinct homotopy types of $S^3$ bundles over a product of Riemann surfaces that admit cone decomposable Sasaki fiber joins, and a countable infinity of distinct homotopy types that admit cone indecomposable Sasaki fiber joins. We say more in the next section for the case $g_1=g_2=0$. Here we mention one further topological invariant, namely the  second Stiefel-Whitney class $w_2(M_\gK)$ which is the mod 2 reduction of $c_1(\cald_\gK)$. Of course $M_K$ is a spin manifold if and only if $w_2(M_K)=0$. We now have
\begin{proposition}\label{spinman}
Let $M_K$ be a Sasaki fiber join on an $S^{3}$ bundle over $\grS_{g_1}\times \grS_{g_2}$. Then 
\begin{enumerate}
\item if both $d_0$ and $d_\infty$ are odd, $w_2(M_K)=0$;
\item if both $d_0$ and $d_\infty$ are even, $w_2(M_K)=0$ if and only if $k^i_0+k^i_\infty$ is even for $i=1,2$;
\item if one of $d_0,d_\infty$ is odd and one is even, $w_2(M_K)=0$ if and only if both components of one of the column vectors of $K$ are even.
\end{enumerate}
\end{proposition}

\begin{proof}
$w_2(M_K)$ is the mod 2 reduction of $c_1(\cald_K)$ which from Proposition \ref{c1yam} and Equation \eqref{kmatrixeqn2} is
\begin{eqnarray}\label{c1noncoeqn} 
c_1(\cald_K) &=&\bigl(2-2g_1-(d_0+1)k^1_0-(d_\infty+1)k^1_\infty)\bigr)\pi^*_M[\grO_1] \\ \notag
                   &+&\bigl(2-2g_2-(d_0+1)k^2_0-(d_\infty+1)k^2_\infty)\bigr)\pi^*_M[\grO_2].
\end{eqnarray}
So the result easily follows.
\end{proof}

We also have the following corollary of Propositions \ref{k>1topprop} and \ref{spinman}
\begin{corollary}\label{z2cor}
Let $M_K$ be a Sasaki fiber join on an $S^{3}$ bundle over $\grS_{g_1}\times \grS_{g_2}$. Then 
\begin{enumerate}
\item if $e$ is odd $M_K$ is cone indecomposable and non-spin, 
\item if $e=2$ $M_K$ is cone decomposable with $k^1_0=k^1_\infty=k^2_0=k^2_\infty=1$. 
\end{enumerate}
\end{corollary}

\begin{proof}
Since $e=k^1_0k^2_\infty+k^2_0k^1_\infty$ and $M_K$ is cone decomposable if and only if $0=\det K=k^1_0k^2_\infty-k^2_0k^1_\infty$, we have $e=2k^1_0k^2_\infty$ if and only if $M_K$ is cone decomposable. Then using (2) of Proposition \ref{spinman} shows that $e$ must be even, so $w_2\neq 0$ which proves (1). (2) is clear.
\end{proof}

\begin{remark}\label{torrem}
When $d=1$ there is a classification of such sphere bundles due to Dold and Whitney \cite{DoWh59}.  For any $d$ there can be topological twists such as $\grS_{g_1}\times \bigl(\grS_{g_2}\tilde{\times} S^{2d+1}\bigr)$. The so-called cohomological rigidity problem is of much interest; however, most of what is known occurs in even dimension \cite{ChMaSu11}, and lies beyond the scope of the present article. 
\end{remark}

\subsubsection{Extremal examples and proof of parts (1) and (2) of Theorem \ref{noncothm}}
As a particular example we take $N$ to be a product of the Riemann sphere with a Riemann surface of arbitrary genus $g$. We can consider three cases $g=0,1,>1$. In the first two we can be quite general; however, when $g>1$ complications arise so we specialize further. We now give a result that proves parts (1) and (2) of Theorem \ref{noncothm}. 

\begin{proposition}\label{speccaseprop}
Let $M_K$ be an admissible fiber join with $N=\bbc\bbp^1\times \grS_g$. Then
\begin{enumerate}
\item if $g=0,1$ the regular ray generated by $\xi_1$ in $\gt^+(\cald_K,J)$ has an extremal representative for all $d_0,d_\infty$ and all $K\in\gK$;
\item if $g>1$, $d_0=d_\infty=1$ and we choose 
\begin{equation*}\label{specextrmatrix2}
K= \begin{pmatrix}
     2 & g \\
     1 & 1
\end{pmatrix},
\end{equation*}
the regular ray in $\gt^+(\cald_K,J)$ has an extremal representative.
\end{enumerate}
Thus, in both cases the Sasaki cone $\gt^+(\cald_K,J)$ of $M_\gw$ as well as the subcone $\gt^+_{sph}(\cald_K,J)$ contains an open set of extremal Sasaki metrics.
\end{proposition}

\begin{proof}
If $g=0$ the quotient complex manifold of $M_K$  arising from the regular Sasakian structure with Reeb vector field $\xi_1$ is equal to 
\begin{eqnarray}\label{KSprojeqn} \notag
\bbp\bigl(L_0\otimes \bbc^{d_0+1}\oplus L_\infty \otimes \bbc^{d_\infty+1}) = \bbp\bigl(\bbc^{d_0+1} \oplus L^*_0\otimes L_\infty\otimes \bbc^{d_\infty+1}\bigr) \\  
= \bbp\bigl(\bbc^{d_0+1}\oplus \calo(k^1_0-k^1_\infty,k^2_0-k^2_\infty)\otimes \bbc^{d_\infty+1}\bigr) \rightarrow \bbc\bbp^1\times\bbc\bbp^1. 
\end{eqnarray}
So the projectivizaton $\bbp\bigl(\bbc^{d_0+1} \oplus L^*_0\otimes L_\infty\otimes \bbc^{d_\infty+1}\bigr)$ is a $\bbc\bbp^d$ bundle over $\bbc\bbp^1\times\bbc\bbp^1$ which is a generalized Bott manifold \cite{ChMaSu10}, and we have the following commutative diagram
\begin{equation}\label{bundiag2}
\xymatrix{
S^1\ar[d]^{id} \ar[r]  &S^{2d+1}\ar[d] \ar[r] &\bbc\bbp^d \ar[d]  \\
               S^1 \ar[r]&M_K \ar[r] \ar[d] &\bbp\bigl(\bbc^{d_0+1}\oplus \calo(k^1_0-k^1_\infty,k^2_0-k^2_\infty)\otimes \bbc^{d_\infty+1}\bigr).\ar[d]  \\
& \bbc\bbp^1\times \bbc\bbp^1 \ar[r]^{id}  &\bbc\bbp^1\times \bbc\bbp^1}
\end{equation}
Now, every K\"ahler class on $\bbp\bigl(\bbc^{d_0+1}\oplus \calo(k^1_0-k^1_\infty,k^2_0-k^2_\infty)\otimes \bbc^{d_\infty+1}\bigr)$ is admissible in the broader sense of the definition given in \cite{ACGT08}, (so the fiber join is super admissible). Thus, for $g=0$ the result  follows from Proposition \ref{gennonco}.  
        When $g\geq 1$ we note that by Remark 2 in [ACGT08] (or the comments above Definition 4.1) that the fiber joins are still super admissible. The $g=1$ case then follows from Proposition 4.3, similarly to above, and the special case with $g>1$ follows from Proposition \ref{appenprop} in the Appendix.
\end{proof}

\begin{remark}\label{genusrem}
When $g>1$ we expect many other examples. However, the analysis in the Appendix that assures the positivity of the `extremal polynomial' needs to be worked out in each case which becomes more complicated as $d_0,d_\infty$ grow. 
Note also that for the example given in (2) of Proposition \ref{speccaseprop} if $g>2$ the fiber join $M_K$ is cone indecomposable, whereas, it is cone decomposable when $g=2$. 
\end{remark}

\subsubsection{CSC examples and proof of part (3) of Theorem \ref{noncothm}}
In the above setting we will now assume that $N=\grS_{g_1}\times \grS_{g_2}$ and $d=1$. The quotient of the regular ray equals
$\bbp(\BOne \oplus L^*_0\otimes L_\infty)$ with $c_1(L_0)=-[\omega_0]$, $c_1(L_\infty)=-[\omega_\infty]$, and $\omega_0$ and $\omega_\infty$ are given by \eqref{kmatrixeqn2}. In order to apply the admissible construction we assume $k^i_0\neq k^i_\infty$ for $i=1,2$. We will use our findings in Appendix \ref{kahclass} together with results from Section 3.4 of \cite{ACGT08} to obtain admissible cone indecomposable fiber joins with a cscS regular ray (up to isotopy). Note that in the cone decomposable case, i.e., the colinear case, no true admissible case ($[\omega_0]\neq [\omega_\infty]$) would have a cscS regular ray (see e.g. Theorem 7 of \cite{ACGT08}), but in that case we have a regular join and so by Theorem \ref{JGAthm} we have a cscS ray elsewhere in the Sasaki cone. 
Thus, below we assume non-colinearity.

Here we have 
$$\omega_0-\omega_\infty= (k^1_0-k^1_\infty)\grO_1 +(k^2_0-k^2_\infty)\grO_2$$
and
$$\omega_0+\omega_\infty= (k^1_0+k^1_\infty)\grO_1 +(k^2_0+k^2_\infty)\grO_2$$
so \eqref{xassumption} in Appendix A is satisfied with $\cala=\{1,2\}$,
$\omega_{N_1}= 2\pi (k^1_0-k^1_\infty)\grO_1$, $\omega_{N_2}= 2\pi(k^2_0-k^2_\infty)\grO_2$.
This means that the K\"ahler class of the regular quotient is admissible with the admissible data associated to the matrix 
$K= \begin{pmatrix}
     k^1_0 & k^2_0 \\
k^1_{\infty} & k^2_{\infty}
\end{pmatrix}$
given by
\begin{equation}\label{admdata}
s_1=\frac{2(1-g_1)}{k^1_0-k^1_\infty},\quad s_2=\frac{2(1-g_2)}{k^2_0-k^2_\infty}, \quad r_1=\frac{k^1_0-k^1_\infty}{k^1_0+k^1_\infty}, \quad r_2=\frac{k^2_0-k^2_\infty}{k^2_0+k^2_\infty}.
\end{equation}
Note that $r_1\neq r_2$ since we are assuming non-colinearity and we have a genuine $|\cala|=2$ admissible case.
Compared to Section 3.4 of \cite{ACGT08} we have a slight notation change (using $r_i$ instead of $x_i$), but aside from this we can directly 
use the findings in this section to get the following slight generalization of Lemma 7 of \cite{ACGT08}

\begin{lemma}\label{ACGT08lemma7}
With $s_1,s_2,r_1,r_2$ given by the admissible data \eqref{admdata}, the K\"ahler class of the regular quotient has an admissible CSC K\"ahler metric iff
\begin{equation}\label{(17)ACGT08}
r_1(s_1(r_1-r_2)-2+(1-s)r_1r_2)+3(s-1)r_2=0,
\end{equation}
\begin{equation}\label{(18)ACGT08}
r_2(s_2(r_2-r_1)-2+(1-s)r_1r_2)+3(s-1)r_1=0,
\end{equation}
for some $s\in \bbr$ (up to scale, $s$ is the scalar curvature of the resulting K\"ahler metric), and
\begin{equation}\label{positivityQ}
Q(\gz)=(1+r_1\gz)(1+r_2\gz)+(1-\frac{s}{2})r_1r_2(1-\gz^2)>0, \,\, \text{for} \,\, -1<\gz<1.
\end{equation}
Moreover, if $s\geq 0$ condition \eqref{positivityQ} automatically holds.
\end{lemma}

\begin{proof}
From Section 3.4 of \cite{ACGT08} we see, assuming $r_1\neq r_2$, that a CSC K\"ahler metric exists if and only if $s$ satisfies Equations \eqref{(17)ACGT08} and \eqref{(18)ACGT08} and the extremal polynomial $F_\grO(\gz)$ of \cite{ACGT08}, which equals $(1-\gz^2)Q(\gz)$, is positive on $-1<\gz<1$. This proves the first statement.
For the second statement we note that since $s_a r_a<2$, equations \eqref{(17)ACGT08} and \eqref{(18)ACGT08} imply that $r_1r_2<0$, i.e., $(k^1_0-k^1_\infty)(k^2_0-k^2_\infty)<0.$ For a solution to \eqref{(17)ACGT08} and \eqref{(18)ACGT08} with $s\geq 0$, we see right away that $Q(\gz)$ is concave down or linear. Hence, since $Q(\pm 1)>0$, \eqref{positivityQ} holds.
\end{proof}

\bigskip

While a systematic analysis of all the possible solutions of \eqref{(17)ACGT08} and \eqref{(18)ACGT08} with \eqref{admdata} seems not really to be tractable, we can obtain solutions by making a special ansatz. In particular, the first part of Lemma 8 of \cite{ACGT08} is:

\begin{lemma}\cite{ACGT08}\label{ACGT08lemma}
With the admissible data \eqref{admdata} assumed, suppose in addition that $s_1+s_2=0$ and $r_1+r_2=0$. Then
$s=\frac{1-r_1^2+2s_1r_1}{3-r_1^2}$ is a solution to Equations \eqref{(17)ACGT08} and \eqref{(18)ACGT08}. In particular, if in addition $g_1=0,1$, then $s\geq 0$ giving a CSC K\"ahler metric.
\end{lemma}

\begin{proof}
The first statement is precisely the first case of Lemma 8 in \cite{ACGT08}, and the second statement follows from Lemma \ref{ACGT08lemma7} together with $s_1r_1=\frac{2(1-g_1)}{k^1_0+k^1_\infty}$ since $s\geq 0$ will hold if $s_1r_1\geq 0$.
\end{proof}

Note that the conditions in the hypothesis of Lemma \ref{ACGT08lemma} imply that $g_1=1$ if and only if $g_2=1$, and since $s_a=0$ the only other condition is $r_1+r_2=0$. However, if $g_1=g_2=0$ the conditions $s_1+s_2=r_1+r_2=0 $ imply $k^1_0=k^2_\infty$ and $k^1_\infty=k^2_0$. So for $N=\bbc\bbp^1\times \bbc\bbp^1$ and $N=T^2\times T^2$ Lemma \ref{ACGT08lemma} gives

\begin{lemma}\label{g1g0lem}
Assume the hypothesis of Lemma \ref{ACGT08lemma}. If
\begin{enumerate}
\item $g_1=g_2=0$ and $K= \begin{pmatrix}
    k & l \\
l & k
\end{pmatrix}$
with $k,l\in \bbz^+$ and $k\neq l$, or,
\item  $g_1=g_2=1$,
and $K= \begin{pmatrix}
     k^1_0 & k^2_0 \\
k^1_{\infty} & k^2_{\infty}
\end{pmatrix}$ is such that 
$\frac{k^1_0-k^1_\infty}{k^1_0+k^1_\infty}=\frac{k^2_\infty-k^2_0}{k^2_0+k^2_\infty}$;
\end{enumerate}
then, up to isotopy, the K\"ahler metric is CSC.
\end{lemma}

When $g_i>1$ we can still look for solutions with positive scalar curvature, but generally it is somewhat tedious. So we make the additional assumption that the genera are equal.

\begin{lemma}\label{ACGT08lemma8c}
Under the hypothesis of Lemma \ref{ACGT08lemma} with $g_1=g_2=g$,
if $K= \begin{pmatrix}
     k+1 & k \\
k & k+1
\end{pmatrix}$ is such that $k\in \bbz^+$ satisfies
$k> \lfloor\frac{2 g-3 +\sqrt{4 g^2-8 g+5}}{2} \rfloor$,
then, up to isotopy, the K\"ahler metric is CSC.
\end{lemma}

\begin{proof}
Assuming $g_1=g_2=g$ and setting $k_0^1=k+1$, $k_\infty^1=k$, $k_0^2=k$, and $k_\infty^2=k+1$ for $k\in\bbz^{+}$, we have
$s_1=2(1-g)$, $s_2=2(g-1)$, $r_1=\frac{1}{2k+1}$, and $r_2=\frac{-1}{2k+1}$. Thus, by Lemma \ref{ACGT08lemma}, we know  that
 \eqref{(17)ACGT08} and \eqref{(18)ACGT08} are solved with $s=\frac{2(k^2+(3-2g)k+(1-g))}{6k^2+6k+1}$. We can now use Lemma \ref{ACGT08lemma7}, for any value of $g$, as long as $k^2+(3-2g)k+(1-g)>0$, that is, as long as $k> \lfloor\frac{2 g-3 +\sqrt{4 g^2-8 g+5}}{2} \rfloor$.
\end{proof}

Putting these lemmas together gives Theorem \ref{noncothm} of the Introduction. 

To end this subsection let us give a couple of examples of solutions to \eqref{(17)ACGT08}, \eqref{(18)ACGT08},  and \eqref{positivityQ}, where the transverse scalar curvature $s=s^T_g<0$:

\begin{proposition}\label{randomcscS}
In the above setting with $g_1= 5$, $g_2=3$, and
 $K= \begin{pmatrix}
     2 & 1 \\
1 & 3
\end{pmatrix}$, up to isotopy, the regular ray is cscS.
\end{proposition}

\begin{proof}
From \eqref{admdata}, we have
$$s_1=-8,\quad s_2=2, \quad r_1=1/3, \quad r_2=-1/2,$$
which is easily seen to solve \eqref{(17)ACGT08} and \eqref{(18)ACGT08} with $s=-1$. Now,
$Q(\gz)$ from \eqref{positivityQ} is given by $Q(\gz) = (1/12) (9 - 2 \gz + \gz^2)$, which clearly satisfies the condition in \eqref{positivityQ}.
\end{proof}

\begin{proposition}\label{randomcscS2}
In the above setting with $g_1= 18$, $g_2=14$, and
 $K= \begin{pmatrix}
     2 & 1 \\
1 & 3
\end{pmatrix}$, up to isotopy, the regular ray is cscS.
\end{proposition}

\begin{proof}
From \eqref{admdata}, we have
$$s_1=-34,\quad s_2=13, \quad r_1=1/3, \quad r_2=-1/2,$$
which is easily seen to solve \eqref{(17)ACGT08} and \eqref{(18)ACGT08} with $s=-6$. Now,
$Q(\gz)$ from \eqref{positivityQ} is given by $Q(\gz) = (1/6) (2 -  \gz +3 \gz^2)$, which clearly satisfies the condition in \eqref{positivityQ}.
\end{proof}

\begin{remark}
Note that with $K$ as in Propositions \ref{randomcscS} and \ref{randomcscS2} we can also solve
\eqref{(17)ACGT08} and \eqref{(18)ACGT08}  with  \eqref{admdata} when $g_1= 31$, $g_2=25$, and $s=-11$. Here, however,
$Q(\gz)=  (1/12) (-1 - 2 \gz + 11 \gz^2)$, and clearly \eqref{positivityQ}  fails. 
\end{remark}

\bigskip

\subsection{$N=\bbc\bbp^1\times \bbc\bbp^1$ and the Proof of Theorem \ref{gen0prodcor}}
We now consider the case that $N$ is a product of Riemann spheres in more detail. Again we divide our analysis into the two cases $d>1$ and $d=1$. 
\subsubsection{$d>1$}
We are interested in the total Pontrjagin class of the manifolds $M_\gK$. It follows from Proposition \ref{k>1topprop} that the only non-vanishing Pontrjagin class of $M_\gK$ is $p_1$ which we note is not only a diffeomorphism invariant but also a homeomorphism invariant\footnote{This appeared to be a folklore result, but Diarmuid Crowley provided us the reference \cite{Sha85} for its proof.} \cite{Sha85}. So Sasaki CR manifolds $M_\gw$ with distinct $p_1$ cannot be homeomorphic.

\subsubsection{Proof of Theorem \ref{gen0prodcor} when $d>1$.}
First it follows from Remark \ref{totchern} that when $d>1$ $p_1(M_\gK)=c_1(\cald_\gK)^2-2c_2(\cald_\gK)$ is equal to
\begin{eqnarray} 
                   &=& c_1(\cald_\gK)^2 - 2\pi^*_M\bigl(\grs_2(-[\gro_0],\ldots,-[\gro_0],-[\gro_\infty],\ldots,-[\gro_\infty])             +c_2(\bbc\bbp^1\times\bbc\bbp^1)\bigr)  \notag \\  \notag
                   &=& 2(2-(d_0+1)k^1_0-(d_\infty+1)k^1_\infty)(2-(d_0+1)k^2_0-(d_\infty+1)k^2_\infty)\grg -8\grg   \\  \notag
                   &-& 2\bigl(d_0(d_0+1)k^1_0k^2_0+d_\infty(d_\infty+1)k^1_\infty k^2_\infty +(d_0+1)(d_\infty+1)(k^1_0k^2_\infty +k^1_\infty k^2_0)\bigr)\grg  \\  \notag
                   &=& \bigl(-4(k^1_0+k^2_0+k^1_\infty+k^2_\infty) + 2k^1_0k^2_0+2k^1_\infty k^2_\infty-4d_0k^1_0-4d_0k^2_0 -4d_\infty k^1_\infty  \\     \notag
                   &&  -4d_\infty k^2_\infty +2d_0k^1_0k^2_0+2d_\infty k^1_\infty k^2_\infty\bigr)\grg \label{p1eqn}
 \end{eqnarray} 
where $\grg=\pi^*_M[\grO_1\wedge\grO_2]$. Here we have used
$$\grs_2=\bigl(d_0(d_0+1)k^1_0k^2_0+d_\infty(d_\infty+1)k^1_\infty k^2_\infty +(d_0+1)(d_\infty+1)(k^1_0k^2_\infty +k^1_\infty k^2_0\bigr)[\grO_1\wedge\grO_2].$$

\begin{proposition}\label{infcontman}
Consider the set $\{M_K\}$ of  fiber joins with $K= \begin{pmatrix}
    k_0 & 2 \\
   k_\infty & 2
\end{pmatrix}$, $k_0,k_\infty\in\bbz^+$, $d>1$,
and fixed integral cohomology ring $R_\bbz$. Then there are finitely many diffeomorphism types of both  cone decomposable and cone indecomposable fiber joins within the set $\{M_{K}\}$ whose integral cohomology ring is isomorphic to $R_\bbz$ has a finite number of diffeomorphism types.
Therefore, at least one such Sasaki CR manifold of both types admits a countable infinity of inequivalent contact structures of Sasaki type.
\end{proposition}

\begin{proof}
Notice that if we choose $k^2_0=k^2_\infty=2$ we obtain a two parameter family of sequences $M_{k^1_0,k^1_\infty}$ with constant first Pontrjagin class, namely from the equation above we see that the first Pontrjagin class is
\begin{equation}\label{specialp1}
p_1(M_{k^1_0,k^1_\infty})=-8(d_0+d_\infty+2)\grg.
\end{equation}
So the set $\{M_K\}$ with $K$ as in the Proposition has the same Pontrjagin class $p_1$, and since $d>1$ the Euler class of the sphere bundle vanishes. So by Lemma \ref{forlem} these manifolds are formal. It follows from this and Proposition \ref{k>1topprop} that the hypothesis of Theorem 13.1 of \cite{Sul77} is satisfied. Hence, there are at most a finite number of diffeomorphism types. When $k_0=k_\infty=k$ the Sasaki CR structure is cone decomposable; otherwise, it is cone indecomposable. But also one easily sees from Equation \eqref{c1noncoeqn} that $c_1(\cald_K)$ depends linearly on the parameters $k_0$ and $k_\infty$ giving rise to infinitely many inequivalent contact structures. 
\end{proof}

\subsubsection{$d=1$}  When $d=1$ Sullivan's minimal model is not formal \cite{KrTr91,BFMT16} since a triple Massey product is non-vanishing; however, Corollary 2.3 of \cite{KrTr91} says that, nevertheless, there is a finite number of diffeomorphism types if the equation
\begin{equation}\label{p1=2e}
p_1\equiv 2e\mod 4
\end{equation}
holds.

\subsubsection{Proof of Theorem \ref{gen0prodcor} when $d=1$.}

\begin{proposition}\label{infcontman2}
Consider the set $\{M_K\}_{k_1,k_2}$ of  fiber joins with $K=\begin{pmatrix}
     k_1+b &k_2+c \\
     k_1 & k_2
\end{pmatrix}$, $d=1$,
and fixed integral cohomology ring $R_\bbz$. Then for any fixed pair $(b,c)$ there exists a countable subsequence $\{M_k\}\subset \{M_K\}$ of cone indecomposable fiber joins within the set $\{M_{k}\}$ and finitely many diffeomorphism types whose integral cohomology ring is isomorphic to $R_\bbz$. Therefore, at least one such diffeomorphism class admits a countable infinity of inequivalent contact structures of Sasaki type for both spin and non-spin manifolds.
\end{proposition}

\begin{proof}
Let $M_0,M_\infty$ be the Sasaki 5-manifolds given by the $S^1$ bundles over $\bbc\bbp^1\times \bbc\bbp^1$ corresponding to the K\"ahler classes $\gro_0,\gro_\infty$ of Equation \eqref{d0dinftyeqns}, and consider 7 dimensional  the fiber join $M_0\star_f M_\infty$ with Sasaki CR structures parameterized by the Neron-Severi lattice $\calk_{NS}(\bbc\bbp^1\times \bbc\bbp^1)=span_{\bbz^+}\{[\grO_1],[\grO_2]\}$. 
The projectivization of Equation \eqref{KSprojeqn} by the regular Reeb vector field $\xi_1$ is the twist 1 stage 3 Bott manifold $M_3(0,k^1_0-k^1_\infty,k^2_0-k^2_\infty)$ whose K\"ahler geometry was studied in \cite{BoCaTo17}. The cohomology ring is
$$H^*(M_3(0,b,c),\bbz)=\bbz[x_1,x_2,x_3]/\bigl(x_1^2,x_2^2,x_3(bx_1+cx_2+x_3)\bigr)$$
where $x_1,x_2,x_3$ are 2-classes. Since the base is untwisted, the $x_i$ represent the volume form of the two $\bbc\bbp^1$s. Moreover, its first Pontrjagin class  is 
\begin{equation}\label{p1}
p_1(M_3(0,b,c))=2bcx_1x_2.
\end{equation}
where $b=k^1_0-k^1_\infty$ and $c=k^2_0-k^2_\infty$. Since $M_K$ is an $S^1$ bundle over $M_3(0,b,c)$, the $S^1$ action trivializes the vertical piece of the tangent bundle $M_K$. Thus, the characteristic classes of $M_K$ are the pullbacks of the characteristic classes of $M_3(0,b,c)$. So the first Pontrjagin class is
\begin{equation}\label{2p1}
p_1(M_K)=2(k^1_0-k^1_\infty)(k^2_0-k^2_\infty)\pi^*(x_1x_2)=2bc\pi^*(x_1x_2).
\end{equation}

The existence of the sequence $\{M_k\}$ depends on the parity of $b$ and $c$. To satisfy the condition \eqref{p1=2e} we must choose the subsequence $\{M_k\}$ such that 
\begin{equation}\label{pe}
bk_2+ck_1\equiv 0\mod 2. 
\end{equation}
It is easy to see that such a countable sequence $\{M_k\}$ always exists. Now from \eqref{c1noncoeqn} we have
\begin{equation}\label{c1kl}
c_1(\cald_{K_{k,b,c}})=\bigl(2-2k_1+b\bigr)\pi^*_M[\grO_1] +\bigl(2-2k_2+c\bigr)\pi^*_M[\grO_2].
\end{equation}
So for each pair $(b,c)$ there exists a countable infinity of inequivalent Sasaki contact structures. Moreover, the manifolds $M_{k,b,c}$ are spin when $b$ and $c$ are both even, and non-spin otherwise for all such $k$. Since \eqref{p1=2e} is satisfied Corollary 2.3 of \cite{KrTr91} implies that for both spin and non-spin manifolds the sequence $\{M_{k,b,c}\}$ admits only a finite number of diffeomorphism types.
\end{proof}

Then Propositions \ref{infcontman} and \ref{infcontman2} together with part (1) of Theorem \ref{noncothm} prove Theorem \ref{gen0prodcor}. \hfill $\Box$

\begin{remark}\label{noSE}
For the contact invariant $c_1(\cald_{k,b,c})$ there are three cases to consider
\begin{enumerate}
\item $c_1(\cald_K)$ is negative definite which occurs if and only if no column of $K$ is $(1,1)^t$;
\item $c_1(\cald_K)$ is a negative multiple of $\pi^*_M[\grO_i]$ for some $i=1,2$ which occurs if and only if exactly one column of $K$ is $(1,1)^t$;
\item $c_1(\cald_K)=0$ which occurs if and only if $K=\begin{pmatrix} 1 &1 \\ 
                         1  & 1
                           \end{pmatrix}$.
\end{enumerate}

The only possible SE metric in the collection $\gK(M_\gw)$ occurs in case (3). This is the iterated join $S^3\star S^3\star S^3$ as described in \cite{BG00a} and the SE metric is homogenous with a regular Reeb vector field. 
\end{remark}

\begin{proposition}\label{infcon7man}
For each $b,c\in\bbz$ we consider the set $\{M_{K_{k^i,b,c}}\}_{k^1,k^2}$ of  $S^3$ bundles over $\bbc\bbp^1\times \bbc\bbp^1$. Then
\begin{enumerate}
\item $M_{K_{k^i,b,c}}$ is spin if and only if $b$ and $c$ are both even
\item $M_{K_{k^i,b,c}}$ is cone decomposable if and only if $ck^1=bk^2$.
\end{enumerate}
Each type admits a countable infinity of inequivalent  Sasaki contact structures such that the Sasaki CR structure $(\cald_K,J)$ admits an open set of extremal Sasaki structures in $\gt^+_{sph}$. Moreover, among the elements of $\{M_{K_{k^i,b,c}}\}_{k^1,k^2}$ there are only a finite number of diffeomorphism types with the same integral cohomology ring.
\end{proposition}

\begin{proof}
From \eqref{c1kl} $M_K$ are spin if and only if $b$ and $c$ are both even, so infinite sequences of both occur.
Since $c_1(\cald)$ is a contact invariant our construction gives a countable infinity of such structures. Moreover, by Proposition \ref{infcontman} there are finitely many diffeomorphism types among the set $\{M_{K_{k^i,b,c}}\}_{k^1,k^2}$. 
\end{proof}

Consider the cone indecomposable $S^3$ fiber joins described in case (1) of Theorem \ref{noncothm}, namely those with matrix 
\begin{equation}\label{Kkl}
K= \begin{pmatrix}
    k & l \\
l & k
\end{pmatrix}
\end{equation}
with $k,l\in \bbz^+$ and $k\neq l$. Without loss of generality we can take $k>l$. 

\begin{proposition}\label{g0homeo}
Let $\{M_K\}_{k,l}$ be the set of cone indecomposable fiber joins over $\bbc\bbp^1\times \bbc\bbp^1$ with $K$ given by Equation \eqref{Kkl} and $k>l$. Each such pair $(k,l)$ defines a unique homeomorphism type and the diffeomorphism type is defined up to finite ambiguity.
\end{proposition}

\begin{proof}
From \eqref{2p1} we see that 
\begin{equation}\label{p1kl}
p_1(M_K)=-2(k-l)^2\pi^*x_1x_2.
\end{equation}
and from Proposition \ref{k>1topprop} we have
\begin{equation}\label{kle}
\ge=(k^2+l^2)x_1x_2.
\end{equation}
For each such ordered pair, there is a unique pair of cohomology classes $p_1$ and $\ge$, and since $p_1$ is a homeomorphism invariant we have inequivalent homeomorphism types for distinct ordered pairs $(k,l)$. Then Corollary 2.3 of \cite{KrTr91} implies that the set of possible diffeomorphism types is finite since $p_1\equiv 2\ge \mod 4$.

\end{proof}

\begin{remark}\label{homotopyequivrem}
Although distinct pairs $(k,l)$ are never homeomorphic, if they have the same Euler class $\ge$, they could be homotopy equivalent; however, we do not determine this. 
\end{remark}

\subsection{More General $N$}
As seen from Proposition \ref{gennonco}, our results generalize to the case where $N$ is a local product with  nonnegative cscK metrics. So we can replace  $\bbc\bbp^1$ with any dimensional complex projective space (as the second Betti numbers remain equal to one) and consider arbitrary products of them. We summarize this in a statement below.

\begin{proposition}
Let $N=\bbc\bbp^{n_1}\times \cdots \times \bbc\bbp^{n_k}$ or $N=\bbc\bbp^{n_1}\times \cdots \times \bbc\bbp^{n_k} \times T^2$ and let
$$M_\gw=M_{1}*_f \cdots *_fM_{d+1}$$ 
be a fiber join where $M_{j}$ is a Sasaki manifold which as a principal $S^1$ bundle over $N$
has Euler class $[\gro_{0}] \in H^2(N,\bbz)$ for $j=1,\dots,d_0+1$ and Euler class $[\gro_{\infty}] \in H^2(N,\bbz)$ for
$j=d_0+2,\dots,d_0+ d_\infty+2$, where $d=d_0+d_\infty+1$, and the K\"ahler forms $\gro_{0},\gro_{\infty}$ have constant scalar curvature. Then the regular ray of $M_\gw$ is extremal (up to isotopy). Thus, the Sasaki subcone $\gt^+_{sph}$ of $M_\gw$ will always contain an open set of Sasaki extremal structures.
\end{proposition}

\appendix
\section{Admissible Projective Bundles}\label{APB}
The purpose of this Appendix is two-fold. We want to augment Section \ref{admsect} with a description of the so-called admissible extremal K\"ahler metrics from \cite{ACGT08}, appearing in Proposition \ref{acgtpr11}, and we want to state and prove Proposition \ref{appenprop} which we need for part (2) of Proposition \ref{speccaseprop}. Broadly speaking, we are keeping the notation faithful to the work in \cite{ACGT08} and therefore a tiny bit different from Section \ref{admsect}.

We are interested in the {\it admissible projective bundles} described in \cite{ACGT08}. These are projective bundles of the form 
$$\bbp(E_0\oplus E_\infty)\ra{1.8} N$$ 
that satisfy the following conditions:
\begin{itemize}
\item For $\cala \subset \bbn$ $N$ is covered by a product ${\tilde{N}}=\prod_{a\in \cala} N_a$ of simply connected K\"ahler manifolds $(N_a,\pm \Omega_{N_a},g_a)$ of complex dimension $d_a$. The $\pm$ means that either $+g_a$ or $-g_a$ is positive definite. We assume here that $\pm g_a$ has CSC with scalar curvature $\pm 2 d_a s_a$.
\item $E_0$ and $E_\infty$ are holomorphic projectively flat Hermitian vector bundles over $N$ of ranks $(d_0+1)$ and $(d_\infty+1)$, respectively, which satisfy
$$\frac{c_1(E_\infty)}{(d_\infty+1)}-\frac{c_1(E_0)}{(d_0+1)}=\frac{1}{2\pi}[\gro_N]$$
where $\gro_N=\sum_{a\in \cala}\gro_{N_a}$.
\end{itemize}

\begin{remark} 
To connect with the notation in Section \ref{admsect}, we note that $\gro_{N_a}=2\pi\epsilon_a\Omega_a$.
\end{remark}

\subsection{Admissible Metrics}
On admissible projective bundles as defined above we can now construct the {\it admissible metrics},  \cite{ACGT08}. Here we recover the main points of this construction.
 
We 
let $\hat{\cala}$
be the extended index set:
\begin{itemize}
\item $\hat{\cala} = \cala$, if $d_{0}=d_{\infty}=0$.
\item $\hat{\cala} =\cala \cup \{ 0 \}$, if $d_{0}>0$ and $d_{\infty}=0$.
\item $\hat{\cala} =\cala \cup \{ \infty \}$, if $d_{0}=0$ and $d_{\infty} >0$.
\item $\hat{\cala} =\cala \cup \{ 0 \} \cup \{ \infty \}$, if $d_{0} > 0$
and $d_{\infty} >0$.
\end{itemize}
Now we set $r_{0}=1$, $r_{\infty} = -1$, $s_{0} = d_{0} + 1$ and $s_{\infty} = 
-(d_{\infty} +1)$. This will correspond to $(g_{0}, \omega_{0})$ 
(or $(g_{\infty},\omega_{\infty})$) being the induced Fubini-Study structure with scalar curvature 
$d_{0}(d_{0} +1)$ (or $d_{\infty}(d_{\infty}+1)$) on each fiber of
$P(E_{0})$ (or $P(E_{\infty}))$.

Consider the circle action
on $M=\bbp(E_0\oplus E_\infty)\ra{1.8} N$ induced by the standard circle action on $E_0$. It extends to a holomorphic
$\bbc^*$ action. The open and dense set $M_0$ of stable points with respect to the
latter action has the structure of a principal circle bundle over the stable quotient.
The hermitian norm on the fibers induces via a Legendre transform a function
$\gz:M_0\rightarrow (-1,1)$ whose extension to $M$ consists of the critical manifolds
$e_0:= \gz^{-1}(1)=\bbp(E_{0} \oplus 0)$ and $e_\infty:= \gz^{-1}(-1)=\bbp(0 \oplus E_{\infty})$.
Letting $\theta$ be a connection one form for the Hermitian metric on $M_0$, with curvature
$d\theta = \sum_{a\in\hat{\cala}}\omega_{N_a}$, an admissible K\"ahler metric and form are
given (up to scale) by the respective formulas
\begin{equation}\label{g}
g=\sum_{a\in\hat{\cala}}\frac{1+r_a\gz}{r_a}g_a+\frac {d\gz^2}
{\Theta (\gz)}+\Theta (\gz)\theta^2,\quad
\omega = \sum_{a\in\hat{\cala}}\frac{1+r_a\gz}{r_a}\omega_{a} + d\gz \wedge
\theta,
\end{equation}
valid on $M_0$. Here $\Theta$ is a smooth function with domain containing
$(-1,1)$ and $r_a$, $a \in \cala$ are real numbers of the same sign as
$g_{a}$ and satisfying $0 < |r_a| < 1$. The complex structure yielding this
K\"ahler structure is given by the pullback of the base complex structure
along with the requirement $Jd\gz = \Theta \theta$. The function $\gz$ is hamiltonian
with $K= J \,grad\, \gz$ a Killing vector field. In fact, $\gz$ is the moment 
map on $M$ for the circle action, decomposing $M$ into 
the free orbits $M_{0} = \gz^{-1}((-1,1))$ and the special orbits 
$\gz^{-1}(\pm 1)$. Finally, $\theta$ satisfies
$\theta(K)=1$.

Now $g$ is a (positive definite) K\"ahler metric which extends smoothly to all of $M$ if and only if
$\Theta$ satisfies the following positivity and boundary
conditions
\begin{align}
\label{positivity}
(i)\ \Theta(\gz) > 0, \quad -1 < \gz <1,\quad
(ii)\ \Theta(\pm 1) = 0,\quad
(iii)\ \Theta'(\pm 1) = \mp 2.
\end{align}

The K\"ahler class $\Omega_{\bfr} = [\omega]$ of an admissible metric as in \eqref{g} is also called
{\it admissible} and is uniquely determined by the parameters
$r_a$, $a \in \cala$, once the data associated with $M$ (i.e.
$d_a$, $s_a$, $g_a$ etc.) is fixed. Indeed, we have
$$ \Omega_{\bfr} =\sum_{a \in \cala} \frac{[\omega_{N_a}]}{r_a} + \Xi,$$
where $\Xi$ is a certain fixed cohomology class on $M$.
In the event that $d_0=d_\infty=0$, then we have that
$\Xi$ is the Poincare dual of $2\pi(e_0+e_\infty)$.
For a more thorough description of $\Xi$, please consult Section 1.3 of \cite{ACGT08}.

Define a function $F(\gz)$ by the formula $\Theta(\gz)=F(\gz)/p_c(\gz)$, where
$p_c(\gz) = \prod_{a \in \hat{\cala}} (1 + r_a \gz)^{d_{a}}$.
Since $p_c(\gz)$ is positive for $-1<\gz<1$, conditions
\eqref{positivity}
imply the following conditions on $F(\gz)$, which are only necessary
for compactification of the metric $g$:
\begin{align}
\label{positivityF}
(i)\ F(\gz) > 0, \quad -1 < \gz <1,\quad
(ii)\ F(\pm 1) = 0,\quad
(iii)\ F'(\pm 1) = \mp 2p_c(\pm1).
\end{align}

\subsection{Extremal Admissible Metrics} 
From Proposition 1 in \cite{ACGT08} we know that an admissible metric $g$ given by \eqref{g} is extremal precisely when 
$$F''(\gz) = \left(\prod_{a \in \hat{\cala}} (1+r_a \gz)^{d_a -1} \right)P(\gz),$$
where $P(\gz)$ is a polynomial of degree at most $|\hat{\cala}| +1$ (at most $|\hat{\cala}|$ if $g$ has constant scalar curvature) satisfying that
$$P(-1/r_a) = 2 d_a s_a r_a \prod_{j\in \hat{\cala}\setminus \{a\} }\left( 1-\frac{r_j}{r_a} \right).$$

For fixed admissible data and with a choice of allowable parameters $r_a$, $a\in \cala$ (that is, a choice of an admissible K\"ahler class 
$\Omega_{\bfr})$, it is shown in \cite{ACGT08} that the system above has a unique polynomial solution, $F_{extr}(\gz)$, satisfying
conditions (ii) and (iii) from \eqref{positivityF}. Further, $\Theta(\gz)$ given by $F_{extr}(\gz)/p_c(\gz)$ satisfies (ii) and (iii) of \eqref{positivity}.
An admissible extremal K\"ahler metric, with K\"ahler class $\Omega_{\bfr}$, then results exactly when also (i) of \eqref{positivityF} 
is satisfied.

A root counting argument due to 
Guan \cite{Gua95} and Hwang \cite{Hwa94}, and further explored by Hwang-Singer \cite{HwaSi02}), was adopted in \cite{ACGT08}
to prove Proposition \ref{acgtpr11}.

When not all the CSC K\"ahler metrics $\pm g_a$ have non-negative scalar curvature, the root counting argument breaks down and in general there are many well known such examples where $F_{extr}(\gz)$ will not satisfy (i) of \eqref{positivityF} for all admissible K\"ahler classes (see e.g. Remark 7 in \cite{ACGT08}). While it is also known that if $r_a$ is sufficiently small for all $a\in \cala$, then (i) of \eqref{positivityF} \underline{is} satisfied by $F_{extr}(\gz)$; however, for our purpose this is not useful. For the most part, in this paper, the K\"ahler class is unspecified. Thus, we are looking for cases where we can say that every K\"ahler class has an extremal K\"ahler metric. 
Within the current framework, this means that we are looking for cases where we can say that every admissible K\"ahler class has an extremal K\"ahler metric and that every K\"ahler class is admissible. 

When $N$ is a compact Riemann surface of genus at least two, Theorem 6 in \cite{ACGT08}, which we apply to obtain Proposition \ref{gencoprop}, provides some positive results that seems to indicate that larger $d_0$ and/or $d_\infty$ could be of value. 

Now assume that $N=\bbc\bbp^1\times \Sigma_g$, where $\Sigma_g$ is a compact Riemann surface of genus $g$ at least two. 
We set $\cala =\{1,2\}$. Note that all K\"ahler classes are admissible in this case.

From Section 3.1 in \cite{ACGT08} we know that
$$F_{extr}(\gz)= (1-\gz^2)\left( p_c(\gz) + (1+\gz)^{d_0+1}(1-\gz)^{d_\infty+1} q(\gz)\right).$$
where $q(\gz)=c\gz+e$ and $c,e$ are determined by
the conditions above. 
Working this out explicitly and in general is straightforward (albeit messy), but it gets hopelessly complicated to sort out when the admissible data is such that $F_{extr}(\gz)$ satisfies (i) of \eqref{positivityF}. 

However, to illustrate that further existence results are out there, let us for simplicity assume that $s_1=2$, $0<r_1<1$ (so $\omega_{N_1}/(2\pi)$ is a K\"ahler form with primitive K\"ahler class on $\bbc\bbp^1$) and  $s_2=-2$, $0<r_2<1$ (meaning $\omega_{N_2}/(2\pi)$ is a K\"ahler form whose class is $(g-1)$ times a primitive K\"ahler class on $\Sigma_g$). We will also assume that $d_0=d_\infty=1$. 
Now one can calculate that
$$F_{extr}(\gz)= \frac{(1-\gz^2)^2h(\gz)}{3 \left(3 r_1^2 r_2^2-7 r_1^2+8 r_1 r_2-7 r_2^2+35\right)}$$
with
$$
h(\gz) = a_0 + a_1(\gz+1) + a_2 (\gz+1)^2 + a_3 (1+\gz)^3,
$$
where
\tiny
$$
\begin{array}{ccl}
a_0 
 &=&   3 (1-r_1) (1-r_2) \left(3 r_1^2 r_2^2-7 r_1^2+8 r_1 r_2-7 r_2^2+35\right)\\
\\
a_1 &=&  (1-r_2)\left(12 r_1^3 r_2^2+15 r_1^3 r_2+7 r_1^3 +105 r_1+49 r_2^2+105 r_2\right)\\
\\
& -&  (1-r_2)\left(21 r_1^2 r_2^2+13 r_1^2 r_2+56 r_1^2+48 r_1 r_2^2+91 r_1 r_2\right)\\
\\
a_2 &=& 2 \left(3 r_1^3 r_2^3+3 r_1^3 r_2^2+8 r_1^3 r_2+2 r_1^2 r_2^2
+14 r_1^2+49 r_1 r_2 +7 r_2^3+14 r_2^2\right) \\
\\
&-& 2\left(7 r_1^3+3 r_1^2 r_2^3+30 r_1^2 r_2
+22 r_1 r_2^3+30 r_1 r_2^2\right) \\
\\
a_3 &=& 10 r_1 r_2 (2-r_1+r_2) (r_1+r_2)
\end{array}
$$
\normalsize

Since all the coefficients, $a_i$, for the powers of $(1+\gz)$ are positive (bear in mind that $0<r_i<1$), we see that $h(\gz)>0$ for $\gz>-1$ and hence
$F_{extr}(\gz)>0$ for $-1<\gz<1$. In conclusion, we have the following statement

\begin{proposition}\label{appenprop}
Let $P$ be the admissible manifold $P=\bbp(E_0\oplus E_\infty)\ra{1.8} \bbc\bbp^1\times\Sigma_g$, where $\Sigma_g$ is a compact Riemann surface of genus $g>1$. Let $\Omega_1$ denote the standard Fubini-Study metric on $\bbc\bbp^1$ (with volume one) and let
$\Omega_2$ denote the standard K\"ahler area form on $\Sigma_g$ (with scalar curvature $8\pi(1-g)$).
Assume
$E_0=L_0 \times \bbc^2$, $E_\infty=L_\infty \times \bbc^2$, $c_1(L_0)=[-\omega_0]$, and $c_1(L_\infty)=[-\omega_\infty]$, where
$\omega_0=2\Omega_1+g\Omega_2$ and $\omega_\infty =  \Omega_1+\Omega_2$.
Then every K\"ahler class on $P$ admits an extremal (admissible) metric.
\end{proposition}

\subsection{Regular transverse K\"ahler class: $d=1$ admissible case}\label{kahclass}

At this point we will not attempt to determine the K\"ahler class of the regular quotient of a fiber join in general. However, in the special case of an admissible fiber join with $d=1$, we can extract this information from \eqref{transsympl}. We shall see that in this case -with one additional assumption added- the K\"ahler class is indeed admissible (up to scale).

So assume we have a $d=1$ fiber join as defined in Theorem \ref{Yamthm} and the complex manifold arising as the quotient of the regular Reeb vector field $\xi_1$ 
is equal to $\bbp\left(L^*_1\oplus L^*_2\right) \rightarrow N$ where the base $N$ is a local product of K\"ahler manifolds $(N_a,\Omega_a)$, $a \in \cala \subset \bbn$, where $\cala$ is a finite index set.
For convenience set $L^*_1= L_0$ and $L^*_2= L_\infty$, where $L_0$ and $L_\infty$ are some holomophic line bundles, and set
$\omega_1=\omega_0$ and $\omega_2=\omega_\infty$. Then $c_1(L_0)=-[\omega_0]$ and $c_1(L_\infty)=-[\omega_\infty]$.
Note that for $j=1,2$, the pullback of $\omega_j$ to the Sasaki manifold $M_j$ is $d\eta_j/(2\pi)$ (the factor of $2\pi$ is occasionally ignored or neglected). As part of the admissible assumption we have that
$c_1(L_\infty)-c_1(L_0)= \sum_{a\in \cala}  [\epsilon_a\Omega_a]$, where $\epsilon_a=\pm 1$ or we may write
$c_1(L_\infty)-c_1(L_0)=\frac{1}{2\pi}[\gro_N]$, where $\gro_N=\sum_{a\in \cala}\gro_{N_a}$ and $\gro_{N_a}=2\pi\epsilon_a\Omega_a$.

This means that $2\pi([\omega_0]-[\omega_\infty])=[\gro_N]$. We shall make one additional assumption, namely that
\begin{equation}\label{xassumption}
2\pi([\omega_0]+[\omega_\infty])= \sum_{a\in\cala}\frac{1}{r_a}[\omega_{N_a}],
\end{equation}
for some $r_a$ with same sign as $\omega_{N_a}$ and
$0<|r_a|<1$.

Equation \eqref{transsympl} with $\bfa=(1,1)$ for $d=1$ now reads as
$$d\eta_\BOne= r_1^2d\eta_1 + 2(r_1dr_1\wedge (\eta_1+d\theta_1)) +r_2^2d\eta_2 + 2(r_2dr_2\wedge (\eta_2+d\theta_2))$$

Consider now the transverse K\"ahler structure on 
$\bbp(L_0\oplus L_\infty) \rightarrow N$ (same as $\bbp(\BOne \oplus L_0^*\otimes L_\infty) \rightarrow N$).
We follow the ideas on page 557 in \cite{ACGT08} with no blow-downs:
Since $(r_j,\theta_j)$ are polar coordinates of $L_j^*$ we can say that $z_0:= \frac{1}{2}r_1^2$ and $z_\infty:= \frac{1}{2}r_2^2$  are the moment maps of the natural $S^1$ action on $L_0$ and $L_\infty$, respectively. On $2=z_0+z_\infty$, the function 
$z:= z_0-1=1-z_\infty$ descends to a fiberwise moment map (with range [-1,1]) for the induced $S^1$ action on $\bbp(\BOne \oplus L_0^*\otimes L_\infty) \rightarrow N$. This allows us to rewrite  \eqref{transsympl}  as
$$d\eta_\BOne= 2 (d\eta_1+d\eta_2) + 2d(z\theta),$$
where
$\theta= (\eta_1+d\theta_1) - (\eta_2+d\theta_2)$ is the connection form on $L_0^*\otimes L_\infty$.
According to pages 556-557 in \cite{ACGT08} we have that this descends to a K\"ahler form on $\bbp(\BOne \oplus L_0\otimes L_\infty^*) \rightarrow N$ with K\"ahler class 
$$2(2\pi([\omega_1]+[\omega_2]) + \Xi),$$ 
where $\Xi$ is the fixed class described above.
Due to our extra assumption we then have that - up to scale - the K\"ahler class is
$$ \sum_{a\in\cala}\frac{1}{r_a}[\omega_{N_a}] + \Xi,$$
with $r_a$ as given in \eqref{xassumption},
which is precisely an admissible K\"ahler class.

\begin{remark}
If colinearity holds on top of the above assumptions, we have according to Proposition \ref{joinsident} that  the fiber join is just a regular join.
Here $\omega_i=b_i \omega_N$. Connecting with the notation in \cite{BoTo14a} we have $l_1(w_1,w_2)=(b_1,b_2)$, $l_2=1$, and therefore in Theorem 3.8 of \cite{BoTo14a}, the regular quotient has $n= b_1-b_2$. In the above setting $\cala$ is a singleton, $[\omega_{N_a}]=2\pi n [\omega_N]$, and 
$2\pi([\omega_1]+[\omega_2]) + \Xi = 2\pi (b_1+b_2) [\omega_N] +\Xi = \frac{1}{r_a}[\omega_{N_a}] + \Xi$
with $r_a=\frac{b_1-b_2}{b_1+b_2}$, as it should be according to (44) and (59) in \cite{BoTo14a}.
\end{remark}

\def\cprime{$'$} \def\cprime{$'$} \def\cprime{$'$} \def\cprime{$'$}
  \def\cprime{$'$} \def\cprime{$'$} \def\cprime{$'$} \def\cprime{$'$}
  \def\cdprime{$''$} \def\cprime{$'$} \def\cprime{$'$} \def\cprime{$'$}
  \def\cprime{$'$}
\providecommand{\bysame}{\leavevmode\hbox to3em{\hrulefill}\thinspace}
\providecommand{\MR}{\relax\ifhmode\unskip\space\fi MR }
\providecommand{\MRhref}[2]{%
  \href{http://www.ams.org/mathscinet-getitem?mr=#1}{#2}
}
\providecommand{\href}[2]{#2}


\begin{thebibliography}{ACGTF08}

\bibitem[ACGTF08]{ACGT08}
Vestislav Apostolov, David M.~J. Calderbank, Paul Gauduchon, and Christina~W.
  T{\o}nnesen-Friedman, \emph{Hamiltonian 2-forms in {K}\"ahler geometry.
  {III}. {E}xtremal metrics and stability}, Invent. Math. \textbf{173} (2008),
  no.~3, 547--601. \MR{MR2425136 (2009m:32043)}

\bibitem[AH06]{AlHa06}
Klaus Altmann and J\"urgen Hausen, \emph{Polyhedral divisors and algebraic
  torus actions}, Math. Ann. \textbf{334} (2006), no.~3, 557--607. \MR{2207875}

\bibitem[BCTF19]{BoCaTo17}
Charles~P. Boyer, David M.~J. Calderbank, and Christina~W.
  T{\o}nnesen-Friedman, \emph{The {K}\"{a}hler geometry of {B}ott manifolds},
  Adv. Math. \textbf{350} (2019), 1--62. \MR{3945589}

\bibitem[BFMT16]{BFMT16}
Indranil Biswas, Marisa Fern\'andez, Vicente Mu\~noz, and Aleksy Tralle,
  \emph{On formality of {S}asakian manifolds}, J. Topol. \textbf{9} (2016),
  no.~1, 161--180. \MR{3465845}

\bibitem[BG00a]{BG00a}
C.~P. Boyer and K.~Galicki, \emph{On {S}asakian-{E}instein geometry}, Internat.
  J. Math. \textbf{11} (2000), no.~7, 873--909. \MR{2001k:53081}

\bibitem[BG00b]{BG00b}
Charles~P. Boyer and Krzysztof Galicki, \emph{A note on toric contact
  geometry}, J. Geom. Phys. \textbf{35} (2000), no.~4, 288--298. \MR{MR1780757
  (2001h:53124)}

\bibitem[BG06]{BG05h}
C.~P. Boyer and K.~Galicki, \emph{Highly connected manifolds with positive
  {R}icci curvature}, Geom. Topol. \textbf{10} (2006), 2219--2235 (electronic).
  \MR{MR2284055 (2007k:53057)}

\bibitem[BG08]{BG05}
Charles~P. Boyer and Krzysztof Galicki, \emph{Sasakian geometry}, Oxford
  Mathematical Monographs, Oxford University Press, Oxford, 2008. \MR{MR2382957
  (2009c:53058)}

\bibitem[BGO07]{BGO06}
Charles~P. Boyer, Krzysztof Galicki, and Liviu Ornea, \emph{Constructions in
  {S}asakian geometry}, Math. Z. \textbf{257} (2007), no.~4, 907--924.
  \MR{MR2342558 (2008m:53103)}

\bibitem[BGS08]{BGS06}
Charles~P. Boyer, Krzysztof Galicki, and Santiago~R. Simanca, \emph{Canonical
  {S}asakian metrics}, Commun. Math. Phys. \textbf{279} (2008), no.~3,
  705--733. \MR{MR2386725}

\bibitem[BHLTF18]{BHLT16}
Charles~P. Boyer, Hongnian Huang, Eveline Legendre, and Christina~W.
  T{\o}nnesen-Friedman, \emph{Reducibility in {S}asakian geometry}, Trans.
  Amer. Math. Soc. \textbf{370} (2018), no.~10, 6825--6869. \MR{3841834}

\bibitem[BM93]{BM93}
A.~Banyaga and P.~Molino, \emph{G\'eom\'etrie des formes de contact
  compl\`etement int\'egrables de type toriques}, S\'eminaire Gaston Darboux de
  G\'eom\'etrie et Topologie Diff\'erentielle, 1991--1992 (Montpellier), Univ.
  Montpellier II, Montpellier, 1993, pp.~1--25. \MR{94e:53029}

\bibitem[BMvK16]{BMvK15}
Charles~P. Boyer, Leonardo Macarini, and Otto van Koert, \emph{Brieskorn
  manifolds, positive {S}asakian geometry, and contact topology}, Forum Math.
  \textbf{28} (2016), no.~5, 943--965. \MR{3543703}

\bibitem[BP14]{BoPa10}
Charles~P. Boyer and Justin Pati, \emph{On the equivalence problem for toric
  contact structures on {$S^3$}-bundles over {$S^2$}}, Pacific J. Math.
  \textbf{267} (2014), no.~2, 277--324. \MR{3207586}

\bibitem[BT82]{BoTu82}
R.~Bott and L.~W. Tu, \emph{Differential forms in algebraic topology}, Graduate
  Texts in Mathematics, vol.~82, Springer-Verlag, New York, 1982. \MR{658304
  (83i:57016)}

\bibitem[BTF13]{BoTo11}
Charles~P. Boyer and Christina~W. T{\o}nnesen-Friedman, \emph{Extremal
  {S}asakian geometry on {$T^2\times S^3$} and related manifolds}, Compos.
  Math. \textbf{149} (2013), no.~8, 1431--1456. \MR{3103072}

\bibitem[BTF14]{BoTo13}
\bysame, \emph{Extremal {S}asakian geometry on {$S^3$}-bundles over {R}iemann
  surfaces}, Int. Math. Res. Not. IMRN (2014), no.~20, 5510--5562. \MR{3271180}

\bibitem[BTF16]{BoTo14a}
\bysame, \emph{The {S}asaki join, {H}amiltonian 2-forms, and constant scalar
  curvature}, J. Geom. Anal. \textbf{26} (2016), no.~2, 1023--1060.
  \MR{3472828}

\bibitem[BTF19a]{BoTo19a}
\bysame, \emph{The ${S}^3$ {S}asaki join construction}, arXiv:1911.11031
  (2019).

\bibitem[BTF19b]{BoTo18c}
\bysame, \emph{Sasaki-{E}instein metrics on a class of 7-manifolds}, J. Geom.
  Phys. \textbf{140} (2019), 111--124. \MR{3923473}

\bibitem[BvC18]{BovC16}
Charles~P. Boyer and Craig van Coevering, \emph{Relative {K}-stability and
  extremal {S}asaki metrics}, Math. Res. Lett. \textbf{25} (2018), no.~1,
  1--19. \MR{3818612}

\bibitem[CMS10]{ChMaSu10}
Suyoung Choi, Mikiya Masuda, and Dong~Youp Suh, \emph{Topological
  classification of generalized {B}ott towers}, Trans. Amer. Math. Soc.
  \textbf{362} (2010), no.~2, 1097--1112. \MR{2551516 (2011a:57050)}

\bibitem[CMS11]{ChMaSu11}
\bysame, \emph{Rigidity problems in toric topology: a survey}, Tr. Mat. Inst.
  Steklova \textbf{275} (2011), no.~Klassicheskaya i Sovremennaya Matematika v
  Pole Deyatelnosti Borisa Nikolaevicha Delone, 188--201. \MR{2962979}

\bibitem[Cos07]{Cos07}
Kevin Costello, \emph{Topological conformal field theories and {C}alabi-{Y}au
  categories}, Adv. Math. \textbf{210} (2007), no.~1, 165--214. \MR{2298823}

\bibitem[DW59]{DoWh59}
A.~Dold and H.~Whitney, \emph{Classification of oriented sphere bundles over a
  {$4$}-complex}, Ann. of Math. (2) \textbf{69} (1959), 667--677. \MR{123331}

\bibitem[FOT08]{FeOpTa08}
Yves F{\'e}lix, John Oprea, and Daniel Tanr{\'e}, \emph{Algebraic models in
  geometry}, Oxford Graduate Texts in Mathematics, vol.~17, Oxford University
  Press, Oxford, 2008. \MR{2403898 (2009a:55006)}

\bibitem[GMSW04]{GMSW04a}
J.~P. Gauntlett, D.~Martelli, J.~Sparks, and D.~Waldram,
  \emph{Sasaki-{E}instein metrics on {$S^2\times S^3$}}, Adv. Theor. Math.
  Phys. \textbf{8} (2004), no.~4, 711--734. \MR{2141499}

\bibitem[Gua95]{Gua95}
Daniel Guan, \emph{Existence of extremal metrics on compact almost homogeneous
  {K}\"ahler manifolds with two ends}, Trans. Amer. Math. Soc. \textbf{347}
  (1995), no.~6, 2255--2262. \MR{1285992 (96a:58059)}

\bibitem[HS02]{HwaSi02}
Andrew~D. Hwang and Michael~A. Singer, \emph{A momentum construction for
  circle-invariant {K}\"ahler metrics}, Trans. Amer. Math. Soc. \textbf{354}
  (2002), no.~6, 2285--2325 (electronic). \MR{1885653 (2002m:53057)}

\bibitem[Hwa94]{Hwa94}
Andrew~D. Hwang, \emph{On existence of {K}\"ahler metrics with constant scalar
  curvature}, Osaka J. Math. \textbf{31} (1994), no.~3, 561--595. \MR{1309403
  (96a:53061)}

\bibitem[KO73]{KoOc73}
S.~Kobayashi and T.~Ochiai, \emph{Characterizations of complex projective
  spaces and hyperquadrics}, J. Math. Kyoto Univ. \textbf{13} (1973), 31--47.
  \MR{47 \#5293}

\bibitem[KT91]{KrTr91}
Matthias Kreck and Georgia Triantafillou, \emph{On the classification of
  manifolds up to finite ambiguity}, Canad. J. Math. \textbf{43} (1991), no.~2,
  356--370. \MR{1113760 (92h:57041)}

\bibitem[Leg11]{Leg10}
Eveline Legendre, \emph{Existence and non-uniqueness of constant scalar
  curvature toric {S}asaki metrics}, Compos. Math. \textbf{147} (2011), no.~5,
  1613--1634. \MR{2834736}

\bibitem[Leg16]{Leg16}
\bysame, \emph{Toric {K}\"ahler-{E}instein metrics and convex compact
  polytopes}, J. Geom. Anal. \textbf{26} (2016), no.~1, 399--427. \MR{3441521}

\bibitem[Leg19]{Leg18}
\bysame, \emph{A note on extremal toric almost {K}\"{a}hler metrics}, Moduli of
  {K}-stable varieties, Springer INdAM Ser., vol.~31, Springer, Cham, 2019,
  pp.~53--74. \MR{3967373}

\bibitem[Ler02]{Ler02a}
E.~Lerman, \emph{Contact toric manifolds}, J. Symplectic Geom. \textbf{1}
  (2002), no.~4, 785--828. \MR{2 039 164}

\bibitem[Ler04a]{Ler04b}
\bysame, \emph{Contact fiber bundles}, J. Geom. Phys. \textbf{49} (2004),
  no.~1, 52--66. \MR{2077244}

\bibitem[Ler04b]{Ler04}
\bysame, \emph{Homotopy groups of {$K$}-contact toric manifolds}, Trans. Amer.
  Math. Soc. \textbf{356} (2004), no.~10, 4075--4083 (electronic). \MR{2 058
  839}

\bibitem[MS74]{MiSt74}
J.~W. Milnor and J.~D. Stasheff, \emph{Characteristic classes}, Princeton
  University Press, Princeton, N. J., 1974, Annals of Mathematics Studies, No.
  76. \MR{0440554 (55 \#13428)}

\bibitem[Noz14]{Noz14}
Hiraku Nozawa, \emph{Deformation of {S}asakian metrics}, Trans. Amer. Math.
  Soc. \textbf{366} (2014), no.~5, 2737--2771. \MR{3165654}

\bibitem[Sha85]{Sha85}
Banwari~Lal Sharma, \emph{Topologically invariant integral characteristic
  classes}, Topology Appl. \textbf{21} (1985), no.~2, 135--146. \MR{813284}

\bibitem[Sul77]{Sul77}
D.~Sullivan, \emph{Infinitesimal computations in topology}, Inst. Hautes
  \'Etudes Sci. Publ. Math. (1977), no.~47, 269--331 (1978). \MR{0646078 (58
  \#31119)}

\bibitem[Tak78]{Tak78}
T.~Takahashi, \emph{Deformations of {S}asakian structures and its application
  to the {B}rieskorn manifolds}, T\^ohoku Math. J. (2) \textbf{30} (1978),
  no.~1, 37--43. \MR{81e:53024}

\bibitem[TT18]{TaTs17}
Hiro~Lee Tanaka and Li-Sheng Tseng, \emph{Odd sphere bundles, symplectic
  manifolds, and their intersection theory}, Camb. J. Math. \textbf{6} (2018),
  no.~3, 213--266. \MR{3855080}

\bibitem[TTY16]{TsTsYau16}
Chung-Jun Tsai, Li-Sheng Tseng, and Shing-Tung Yau, \emph{Cohomology and
  {H}odge theory on symplectic manifolds: {III}}, J. Differential Geom.
  \textbf{103} (2016), no.~1, 83--143. \MR{3488131}

\bibitem[Yam99]{Yam99}
T.~Yamazaki, \emph{A construction of {$K$}-contact manifolds by a fiber join},
  Tohoku Math. J. (2) \textbf{51} (1999), no.~4, 433--446. \MR{2001e:53094}

\end{thebibliography}
\end{document}